\documentclass{article}

\usepackage{graphicx}
\usepackage{amssymb}
\usepackage{amsthm}
\usepackage{amsmath}
\usepackage[a4paper,text={6.0in,9.0in},centering,
            includefoot,foot=0.6in]{geometry}
\usepackage[hypertexnames=false,colorlinks=true,linkcolor=blue,citecolor=blue,bookmarksopen=false,bookmarks=false,
            pdfstartview=XYZ,pdffitwindow=true,pdfcenterwindow=true]{hyperref}

\usepackage{multirow}
\usepackage{subfigure}
\usepackage{rotating}
\usepackage{units}

\providecommand{\edt}{\textcolor{black}}

\allowdisplaybreaks[1]
\newtheorem{thm}{Theorem}[section]

\newtheorem{lem}[thm]{Lemma}

\theoremstyle{definition}

\newtheorem{defn}[thm]{Definition}
\newtheorem{rem}{Remark}

\numberwithin{equation}{section}

\title{On the indefinite Helmholtz equation: \edt{complex stretched} absorbing boundary layers, iterative analysis, and preconditioning}

\author{B. Reps, W. Vanroose, H. bin Zubair}

\begin{document}

%\begin{frontmatter}

\maketitle
\begin{center}
\it{Department of Mathematics and Information Technology, University of Antwerp, Middelheim Campus, Middelheimlaan 1, 2020 Antwerp, Belgium}
\end{center}

\begin{abstract}
\edt{This paper studies and analyzes a preconditioned Krylov solver for Helmholtz problems that are formulated with absorbing boundary layers based on complex coordinate stretching. The preconditioner problem is a Helmholtz problem where not only the coordinates in the absorbing layer have an imaginary part, but also the coordinates in the interior region. This results into a preconditioner problem that is invertible with a multigrid cycle. We give a numerical analysis based on the eigenvalues and evaluate the performance with several numerical experiments. The method is an alternative to the complex shifted Laplacian and it gives a comparable performance for the studied model problems.}
\end{abstract}

%\begin{keyword}
%Complex \edt{stretched grid} (CSG) preconditioner, Complex shifted Laplacian (CSL) preconditioner, Multigrid preconditioning (\edt{MGP}), %Bi-CGSTAB, IDR($s$), Exterior complex scaled \edt{(ECS)} absorbing boundary layers.
%\end{keyword}
%\pacs{02.60.Dc, 02.60.Lj, 02.70.-c, 02.70.Bf, 03.65.-w}
% PACS Classification codes
%02.60.Dc Numerical linear algebra
%02.60.Lj Ordinary and partial differential equations; boundary value problems
%02.70.-c Computational techniques; simulations (for quantum computation, see 03.67.Lx; for computational techniques extensively used in subdivisions of physics, see the appropriate section; for example, see 47.11.-j Computational methods in fluid dynamics)
%02.70.Bf Finite-difference methods
%03.65.-w Quantum mechanics [see also 03.67.-a Quantum information; 05.30.-d Quantum statistical mechanics; 31.30.J- Relativistic and quantum electrodynamics (QED) effects in atoms, molecules, and ions in atomic physics]

%\end{frontmatter}

% main text
%%%%%%%%%%%%%%%%%%%%%%%%%%%%%%%%

\section{Introduction}
\edt{The Helmholtz equation is frequently used to model propagation of waves in applications such as acoustic, seismic and electromagnetic realistic systems. It is less known that the equation is also helpful to understand and predict the reaction rates of fundamental processes in few-body physics and chemistry that are important for many areas of technology. In gas discharge reactors that are used industrially for chemical processing of surfaces, for example, these reaction rates are an essential modeling input \cite{BNGM02}. There is both a scientific interest and an industrial need for accurate prediction of reaction rates \cite{Plasma07}. To predict accurately the reaction rates of these processes it is necessary to solve the multi-dimensional Schr\"odinger equation \cite{RBIM99,VMRM05} that is, for the energy regime of these processes, equivalent to a multi-dimensional Helmholtz equation with outgoing waves boundary conditions and a space dependent wave number. The reaction rate of a particular process is then found as a post-processing step where the fluxes of the outgoing waves, corresponding to a particular reaction, are extracted \cite{MHR01}.}

\edt{These applications, often in diverse fields, have little in common amongst them except the Helmholtz model used in the simulations. Naturally, the numerical solution of the indefinite Helmholtz equation forms an interesting field of research for a widespread scientific community. Two main numerical concerns in this context are the truncation of the infinite physical domain to a finite numerical one mapped on a grid, and the efficient iterative solution of the resulting indefinite discrete linear system. In this paper, we concentrate on the latter of these issues.}

\edt{A hard truncation of the physical domain (in the Dirichlet sense) on a numerically feasible finite boundary, results in waves reflecting back and propagating into the truncated domain. These artificial reflections have to be avoided for an acceptable numerical treatment of the Helmholtz equation. This consideration led to the commonly used boundary conditions by Engquist and Majda \cite{EM77} and Bayliss and Turkel \cite{BT80} who approximated the Sommerfeld radiation condition for homogeneous media at the boundary. B\'erenger avoided reflecting waves by adding \emph{perfectly matched layers} (PML) \cite{B94} to the truncated domain. On this absorbing extension the problem is reformulated in order to damp the solution exponentially. Both techniques have been generalized and fine-tuned by many authors, see e.g.\ \cite{G08} for an overview. Chew and Weedon \cite{CW94} related the PML method to a complex coordinate stretching. They consider the absorbing layer as an analytical continuation of the space domain into the complex plane, where the original equation is preserved. Earlier, in the 70's, linear complex scaling of the space domain was already used as a method to compute atomic resonances in microscopic system described by the Schr\"odinger equation \cite{AC71,BC71}. There a coordinate transformation, $r \rightarrow r\exp(i\theta)$ with a angle $\theta> 0$, was applied on the full domain. In the same decade, Simon introduced \emph{exterior complex scaling} (ECS) by only transforming the boundary region with the transformation $r \rightarrow (R_0-r)\exp(i\theta) + r$, where $R_0$ denotes the start of the boundary region. The purpose of this transformation was to introduce the correct boundary condition for resonant states while avoiding analytical continuation of non-analytical potentials in the Schr\"odinger equation. Later this ECS transformation was used to enforce outgoing wave boundary conditions to atomic break-up problems \cite{RBIM99}. Note that the ECS transformation has a discontinuous first derivative in $R_0$, while in PML the complex stretching is usually introduced as a smooth coordinate transformation.}

\edt{The purpose of this paper is to find efficient iterative solvers for Helmholtz problems equipped with absorbing boundary layers based on complex coordinate stretching. The main iterative challenge in a discrete Helmholtz problem $H_h u_h=b_h$ is the indefiniteness of the discretized operator $H_h$. A powerful way to get around this issue is the use of preconditioned Krylov subspace methods. The original troublesome matrix $H_h$ is multiplied by the inverse of the preconditioning matrix $M_h$, resulting in a new system $M_h^{-1}H_h = M_h^{-1}b_h$ for left preconditioning. The choice of the preconditioner is a trade off between a cheaply invertible matrix $M_h$ and a definite preconditioned system $M_h^{-1}H_h$ with nicely clustered eigenvalues. The former demand can be relaxed by allowing an inexact inversion of the preconditioner $M_h$, e.g.\ a few sweeps of a multigrid method. A successful preconditioner in this setup of \emph{multigrid preconditioning} (MGP) of Krylov methods is the \emph{complex shifted Laplacian} (CSL) developed for Sommerfeld radiation conditions by Erlangga, Vuik and Oosterlee \cite{EVO04}.}

\edt{In this paper we study the effect of complex stretching a part of the domain on the eigenvalues of the discretized Helmholtz equation. We have chosen to study the simplest problem with complex coordinate stretching which is the ECS domain as introduced by Simon and still used for atomic break-up problems. We have analyzed the eigenvalues of the one-dimensional Laplacian discretized on an ECS domain. Then the achieved insights are used to apply the CSL preconditioning idea to two-dimensional Helmholtz problems with ECS boundaries. The theoretical analysis also leads to an alternative family of preconditioners based on different complex stretchings of the numerical grid (CSG). Although ECS is not the most accurate absorbing boundary condition, we believe the insights on the performance of the iterative solver are valid for problems where a smooth complex stretching transformation is introduced.}

The remainder of this paper is organized as follows. In Section \ref{sec:ECS}, we describe the ECS absorbing boundary layer. We use a one-dimensional Laplace problem for theoretical considerations; this reference problem is also described in Section \ref{sec:ECS}. The numerical analysis of the discrete one-dimensional reference problem is given next in Section \ref{sec:NAdisc}. Three important lemmas are given here, the insights from which paved the way for the work in this paper. We use three model problems for experimentation. They possess particular properties, which we detail in Section \ref{sec:MP}. Section \ref{sec:prec} follows, and deals with preconditioning ideas. Multigrid behavior as a solver as well as for approximate preconditioner inversion is discussed. We also calibrate multigrid performance with different components. Numerical validation of all the theoretical insights is given in Section \ref{sec:experiments} where the model problems are solved with multigrid preconditioned Bi-CGSTAB \cite{V92} and IDR($s$) \cite{SG08} (with $s=4$ \edt{and $s=8$}), with the CSL and the CSG preconditioning operators, and accompanied by comparison tables.

\edt{\section{Exterior complex stretched domains} \label{sec:ECS}}
%--------------------
% Intro on ECS
%--------------------
The Helmholtz equation in a homogeneous medium is
\begin{defn}[Helmholtz]
\begin{equation}\label{eq:helmctu0}
Hu \equiv -\left(\triangle +k^2\right) u = \chi \mbox{ \edt{in} } \Omega_0\subseteq\mathbb{R}^d,
\end{equation}
with dimension $d\geq1$.
\end{defn}
where $\chi$ is a source term, $k\in\mathbb{R}$ is called the wave number. The equation describes acoustic wave
problems on an unbounded domain $\Omega_0$. We want to solve the Helmholtz equation \eqref{eq:helmctu0} numerically on a bounded part of the domain $\Omega\subset\Omega_0$, with \edt{absorbing boundary layers}.

For the Helmholtz equation \eqref{eq:helmctu0} restricted to the bounded domain $\Omega$, a popular boundary condition is
the first order Sommerfeld boundary condition, given by,
\begin{equation}\label{eq:sommerfeld1}
\frac{\partial u}{\partial \hat{n}} = -\imath ku \text{ on } \partial\Omega,
\end{equation}
\edt{where $\partial \Omega$ represents the domain boundary and $\hat{n}$, the outward normal. We write $\imath$ for the complex identity.} In \edt{multi-}dimensional Helmholtz problems absorbing boundary layers are preferred over these classical first order Sommerfeld conditions because the latter requires an exact knowledge of the wave number at the boundary. \edt{More important, in higher dimensions condition \eqref{eq:sommerfeld1} suffers from artificial reflections and a higher order version should be applied \cite{EM77,BT80}. Equation \eqref{eq:sommerfeld1} is still useful for the analysis of iterative methods though.} The physical interpretation of absorbing boundary layers is to extend the original domain $\Omega$ with a layer $\Gamma$ of an absorbing material. \edt{In} $\Omega$ the original equation is kept and \edt{in} the layer $\Gamma$ the equation is manipulated to enforce specific boundary conditions on the new boundaries $\partial\Gamma$, through an adapted potential due to a change in the material. This was introduced as a \emph{perfectly matched layer} (PML) by B\'erenger~\cite{B94}. The original PML idea is mathematically equivalent to a particular complex coordinate \edt{stretching} \cite{CW94} in the boundary layers, where the original equation is used in \edt{a} new coordinate system. \edt{In this complex stretching approach we define an analytic continuation on the layers by}
\begin{equation} \label{eq:ecstrans}
 z(x) = \left\{
  \begin{array}{ll}
    x, & \hbox{$x \in \Omega$;} \\
    x+\imath f(x), & \hbox{$x \in \Gamma$,}
  \end{array}
\right.
\end{equation}
with $f \in \mathcal{C}^2$ \edt{the \emph{stretching function},} increasing (e.g. linear, quadratic, \ldots) and $\displaystyle \lim_{x\to\partial\Omega}f(x) = 0$. We denote the image of the layer $\Gamma_z\equiv z(\Gamma)$ and call it the \emph{complex contour}. These robust boundary layers do not \edt{use the wave number explicitly as opposed to the Sommerfeld conditions \eqref{eq:sommerfeld1}} and they can easily be tuned in numerical experiments.
The transformation \eqref{eq:ecstrans} constructs the absorbing complex contour by adding a \emph{complex shift} to the domain extension $\Gamma$\edt{, but it can also be done} by a \emph{complex rotation} of $\Gamma$. For a one-dimensional problem \edt{in} $\Omega = [x_0,r]$ with linear complex \edt{stretching} applied \edt{to} the extension $\Gamma = [r,R]$, with an angle \edt{$\theta$} this is typically
\begin{equation*}
z(x) = (x-r)e^{\imath\theta}+r, \quad\hbox{$x \in \Gamma$},
\end{equation*}
\edt{and was introduced by Simon as \emph{exterior complex scaling} \cite{S79}}.
\edt{On discrete level} the mesh width on the contour $\Gamma_z$ becomes $h_\gamma = he^{\imath\edt{\theta}}$. We point out that although the above definition suffices well for practical purposes, as for the experiments presented in this paper, we will keep the analytic discussion general by using the expression \eqref{eq:ecstrans}.

%--------------------
% Laplace problem
%--------------------
We will study the effect of \edt{a complex stretching} transformation on a simple one-dimensional Laplace problem
\begin{equation}\label{eq:lapctu}
-Lu(x) \equiv -\frac{d^2}{d x^2} u(x) = \chi(x) \mbox{ \edt{in} } [0,1] \subset \mathbb{R},
\end{equation}
with $u(0)=0$ and an absorbing boundary condition in $x=1$. The minus sign is introduced to make the Laplacian positive definite. The homogeneous Helmholtz problem only differs in a constant shift $k^2$ and shares the same eigenvectors. An eigenvalue of the Helmholtz problem is $\lambda_{L}-k^2$ where $\lambda_L$ is an eigenvalue of the negative Laplacian $L$. \edt{However, this minor difference can turn the definite problem into a indefinite one, with major consequences on the behavior of iterative methods such as Krylov methods and multigrid methods.} The numerical results in Section~\ref{sec:experiments} \edt{show} that the results are easily extended to two dimensions, and that the same \edt{preconditioning ideas} can be useful \edt{for different values of $k$} and for non-homogeneous Helmholtz problems.

We now implement the complex stretched boundary layer on the one-dimensional Laplace problem \eqref{eq:lapctu} by adding an extension
$\Gamma = [1,R]$ with $1<R\in\mathbb{R}$ to construct the complex contour $\Gamma_z = z(\Gamma)\subset\mathbb{C}$. In this paper, we
use a linear \edt{coordinate transformation on the layer} so that $\Gamma_z$ is the complex line connecting $1$ and $z(R)\equiv R_z\in\mathbb{C}$ that we will \edt{denote} as $[1,z(R)]=[1,R_z]$. This transforms the problem to
\begin{equation}\label{eq:lapctuecs}
-L u(z) \equiv -\frac{d^2}{d z^2} u(z) = \chi(z) \mbox{ \edt{in} } [0,1]\cup[1,R_z]\subset \mathbb{C},
\end{equation}
with homogeneous Dirichlet boundary conditions \edt{at} $z(0)=0$ and $z(R)=R_z$ (see Figure \ref{fig:ctuecs}). \edt{In the remainder of the paper will refer to this linear stretching \eqref{eq:lapctuecs} as the \emph{exterior complex scaled} (or \emph{stretched}) transformation or in short ECS.}
% fig: fig_ctuecs
\begin{figure}[htbp!]
\begin{center}
\includegraphics[width=12cm]{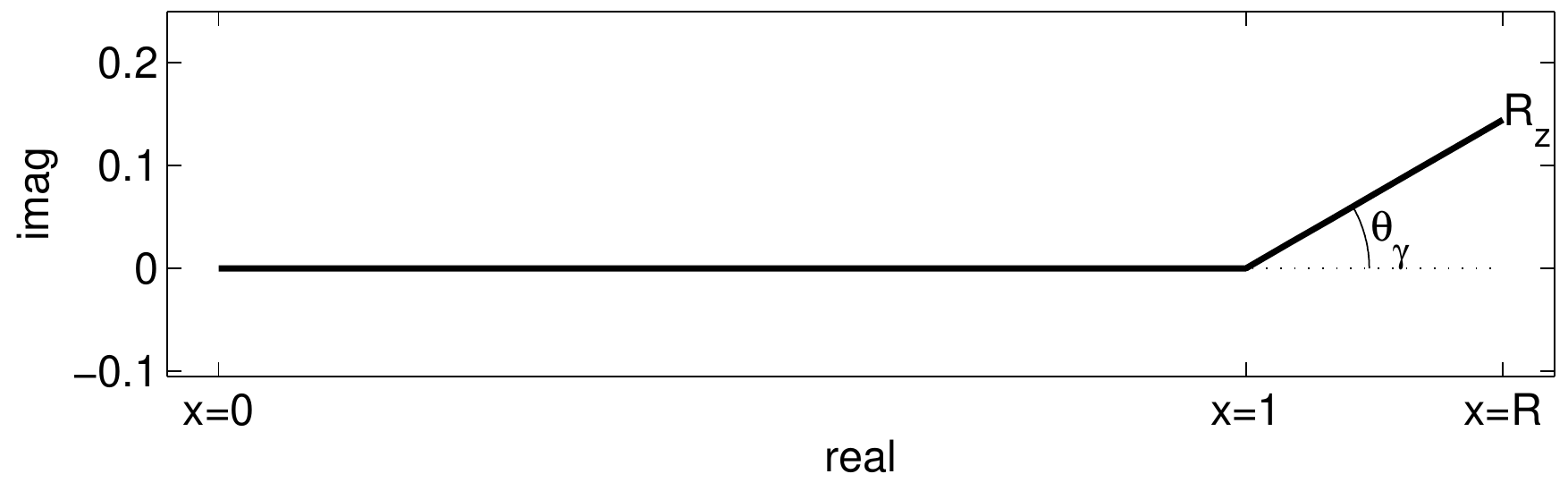}
\caption{The ECS domain $z(x)$. An ECS contour is added as an extension of the domain at $x=1$. This shifts the domain and consequently the spectrum of the resulting operator into the complex plane.}\label{fig:ctuecs}
\end{center}
\end{figure}
The two boundary points $z=0$ and $z=R_z$ determine the position of the eigenvalues independently of the path connecting the two points. This is readily obvious by inspecting the following equation, which gives the eigenvalues of the negative Laplacian on any one-dimensional curve in the complex plane connecting the points $0$ and $R_z$:
\begin{equation*}
\lambda_L = \left(\frac{j \pi}{R_z}\right)^2 \text{ with } j\in\mathbb{N}_0.
\end{equation*}
The derivation is trivial and hence not shown here. The complex contour acts as an absorbing layer. \edt{Indeed, because of the exponential decay of the analytically continued solution one can enforce homogeneous Dirichlet boundary conditions at the end of the complex contour with a significantly smaller boundary error than on a real truncated domain \cite{KP09}, as is illustrated in Figure~\ref{fig:ecsillustration}.
% fig: ecs
\begin{figure}[htbp!]
\begin{center}
\includegraphics[width=12cm]{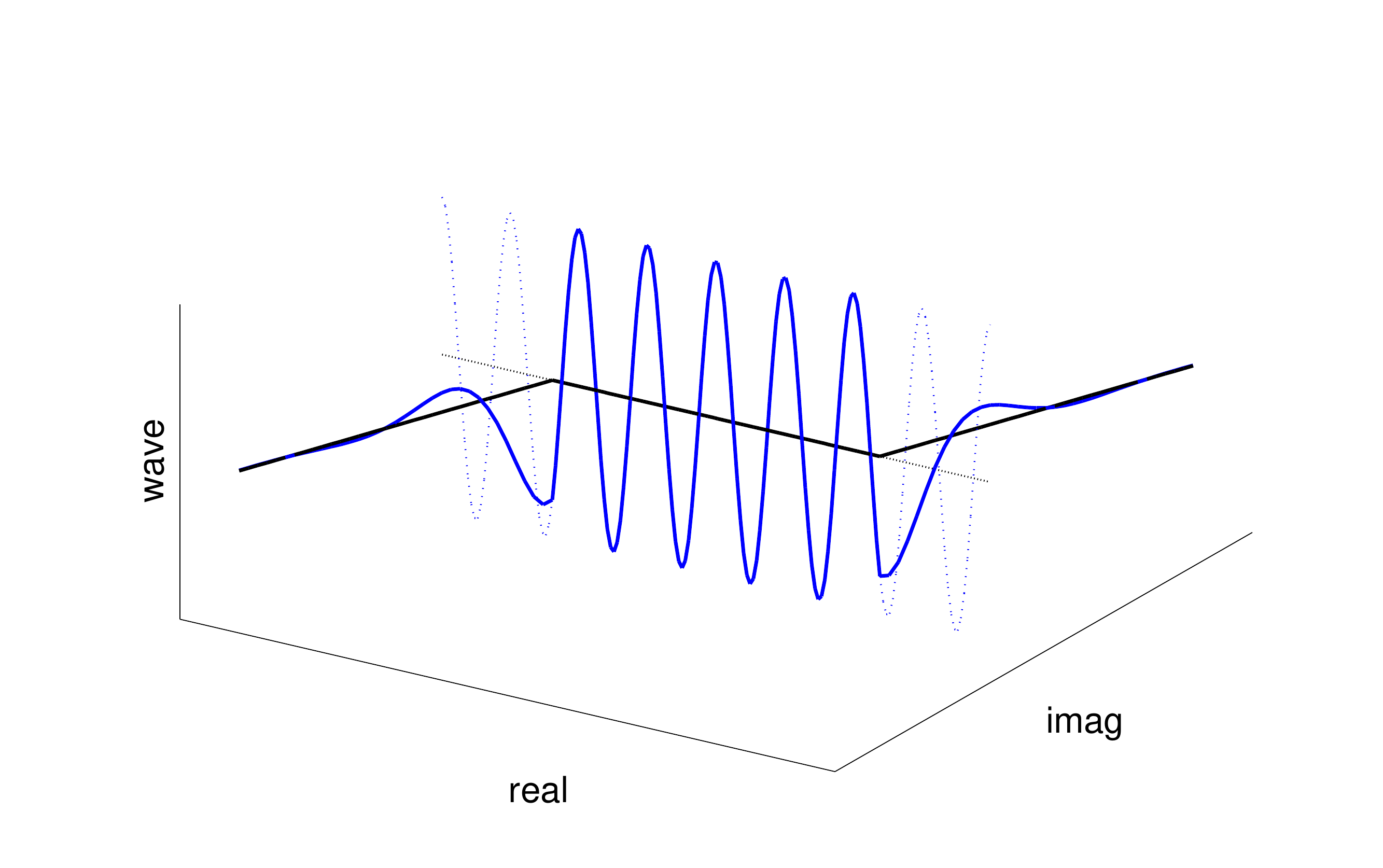} \caption{ECS on a one-dimensional domain. A complex contour is added to both boundaries, enforcing the wave to fulfill Dirichlet boundary conditions. $e^{\imath\kappa z(x)}=e^{\imath\kappa\left(x+\imath f(x)\right)}=e^{\imath\kappa
x}e^{-\kappa f(x)}\approx0$, (color online)}\label{fig:ecsillustration}
\end{center}
\end{figure}
The larger the imaginary part of the complex boundary, the stronger the suppression of the reflected waves. Although the shape of the contour, i.e.\ the stretching function $f(x)$, does not explicitly influence the damping in the continuous case, the discretized problem is susceptible to different shapes, as we will see in the next section.}
\edt{
\begin{rem}[Accuracy of ECS layers]
A detailed discussion on the accuracy of ECS layers does not lie within the scope of this paper; we merely use the simple linear ECS in \eqref{eq:lapctuecs} as a model to understand the iterative solution of the problems with more advanced complex stretched boundary layers.
\end{rem}
}

\section{Numerical analysis of discretized operator} \label{sec:NAdisc}
%------------------------------------
% discretization of Laplace problem
%------------------------------------
In this section we discretize the Laplace problem \eqref{eq:lapctuecs} with finite differences using the
Shortley-Weller formula for non-uniform vertex centered grids \cite{SW38}. It enables discretization through the region of transition to complex mesh widths for the complex contour in the ECS domain. We present theoretical results for the eigenvalues of
the discretization matrix.

Consider the one-dimensional Laplace problem \eqref{eq:lapctuecs}. We define a uniform grid
\begin{equation*}
(z_j)_{0\leq j\leq n} \mbox{ on } [0,1]
\end{equation*}
with $z_0 = 0$ and $z_n=1$ and mesh width $h = 1/n \in\mathbb{R}$, and a second uniform grid on the complex
contour
\begin{equation*}
(z_j)_{n\leq j\leq n+m} \mbox{ on } [1,R_z]
\end{equation*}
with $z_{n+m} = R_z$ and complex mesh width $h_\gamma = (R_z-1)/m$. \edt{We will refer to the angle of $h_\gamma$ in the complex plane as the ECS angle and denote it $\theta_\gamma$.} The union of these two grids is the \edt{ECS} grid
\begin{equation}\label{eq:ecsgrid}
(z_j)_{0\leq j\leq n+m} \mbox{ on } [0,1]\cup[1,R_z]
\end{equation}
in the entire ECS domain. We will often use the fraction \edt{$\gamma = h_\gamma/h\in\mathbb{C}$}. To approximate the second derivative in \eqref{eq:lapctuecs} we choose the Shortley-Weller formula
\begin{equation*}\label{eq:shortwell}
\frac{d^2u}{d z^2}(z_j) \approx
\frac{2}{h_{j-1}+h_j}\left(\frac{1}{h_{j-1}}u_{j-1}-\left(\frac{1}{h_{j-1}}+\frac{1}{h_j}\right)u_j
+\frac{1}{h_j}u_{j+1}\right)
\end{equation*}
for non-uniform grids in grid point $j$, where $h_{j-1}$ and $h_j$ are the left and right mesh widths respectively, and may belong either to the $h$ category or to the $h_\gamma$ category. \edt{The formula is easily derived from Taylor series and} reduces to regular second order central differences when $h_{j-1}=h_j$, i.e., in the interior real region $(0,1)$, and in the interior of the complex contour $(1,R_z)$ because the \edt{stretching} function $f$ is taken to be linear. The only exception is the point $z_n$ where at most we lose an order of accuracy, however with ample discretization steps, the overall accuracy is anticipated to match up to second order.
% fig: fig_disecs
\begin{figure}[htbp!]
\begin{center}
\includegraphics[width=12cm]{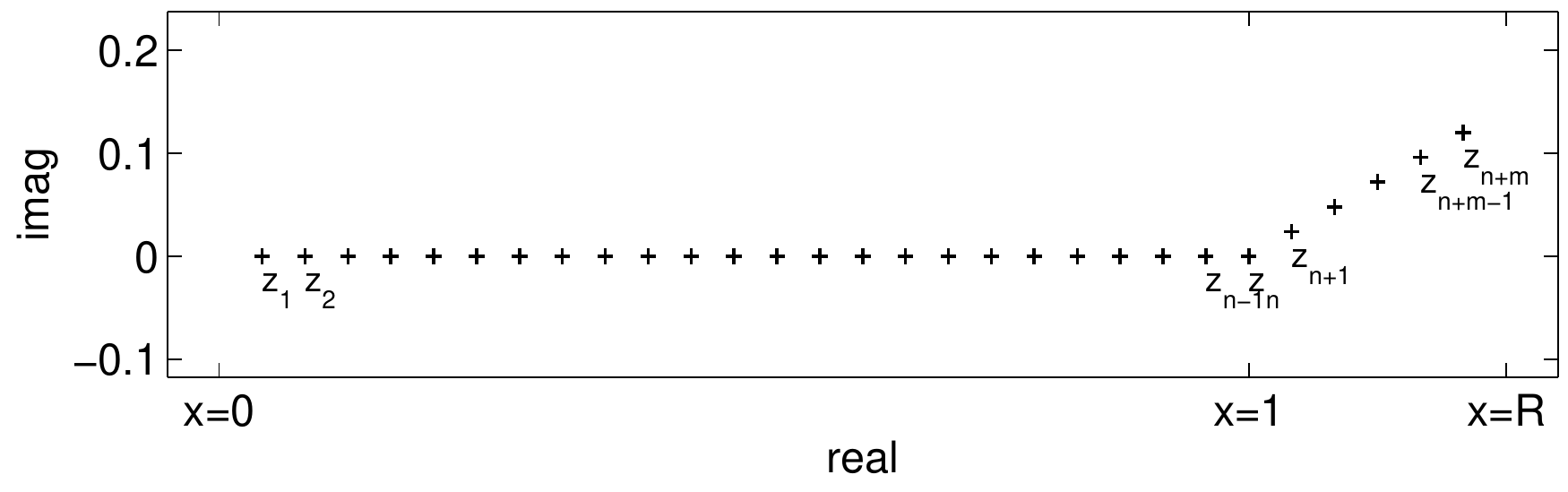}\caption{Discretized ECS domain $z_j$. The ECS domain is discretized with complex mesh widths on the complex contour.}\label{fig:disecs}
\end{center}
\end{figure}
The result is a linear system of equations that we will represent by the matrix equation
\begin{equation}\label{eq:lapdisecs}
-L_h u_h = b_h.
\end{equation}
The right hand side $b_h$ contains contributions from the \edt{source function $\chi$}.
%--------------------------------------
% spectrum of discrete Laplace problem
%--------------------------------------
The spectrum of the discretization matrix $L_h$ in \eqref{eq:lapdisecs} determines the convergence behavior of iterative methods such as Krylov subspace methods and multigrid schemes for solving the system. \edt{It} is drastically different from the spectrum of the continuous operator $L$. We start with the construction of bounds on the field of values of $-L_h$.

The remainder of this section is focused on several lemmas that will help to understand the spectral properties of the discretized operator. First, in Lemma \ref{lem:eigdisbound} the Gershgorin disks are used to produce bounds on the spectrum. Next, in Lemma \ref{lem:eigdis} we find a condition for the eigenvalues of the \edt{discrete ECS Helmholtz operator} for constant wave numbers. The solutions lie on a \edt{pitchfork-shaped figure}. In Lemma \ref{lem:pitchfork} of this section, it is shown how an approximation can be found for the limiting points of this spectrum.
%------------------
% lemma eigdisbound
\begin{lem}\label{lem:eigdisbound}
Consider the ECS grid \eqref{eq:ecsgrid} and the discretization matrix $L_h$ in Equation~\eqref{eq:lapdisecs}. Define $\gamma =
\frac{h_\gamma}{h}$\edt{, and its complex conjugate $\bar{\gamma}$}. If $\lambda\in\sigma(-L_h)$, then
\begin{eqnarray*}
\Re(\frac{1}{\gamma^2})-\left|\frac{1}{2\gamma^2}+\frac{1}{\bar{\gamma}(1+\bar{\gamma})}\right|\leq &h^2\Re(\lambda)& \leq \edt{\max(4,3+\frac{1}{2}\left|\frac{3+\bar{\gamma}}{1+\bar{\gamma}}\right|)}\\
\Im(\frac{4}{\gamma^2})\leq &h^2\Im(\lambda)& \leq \frac{1}{2}\left|\frac{|\gamma|^2-\gamma}{|\gamma|^2+\gamma}\right|\\
\end{eqnarray*}
where $\Re$ and $\Im$ denote the real and imaginary part, respectively.
\end{lem}
\begin{proof}
Every eigenvalue of a matrix lies in the field of values of that matrix. For the field of values \edt{$W(A_N)$} of a matrix \edt{$A_N=(a_{ij})_{1\leq i,j\leq N}$}
holds
\begin{eqnarray*}
\min\{\Re(W(A_N))\} &=& \min\{\mu\in\mathbb{R}:\mu\in\sigma(\frac{A_N+A_N^*}{2})\} \\
\max\{\Re(W(A_N))\} &=& \max\{\mu\in\mathbb{R}:\mu\in\sigma(\frac{A_N+A_N^*}{2})\} \\
\min\{\Im(W(A_N))\} &=& \min\{\mu\in\mathbb{R}:\mu\in-\imath\sigma(\frac{A_N-A_N^*}{2})\} \\
\max\{\Im(W(A_N))\} &=& \max\{\mu\in\mathbb{R}:\mu\in-\imath\sigma(\frac{A_N-A_N^*}{2})\}, \\
\end{eqnarray*}
\edt{with $A_N^*$ the adjoint matrix of $A_N$.} We scale the model problem with $h^2$ and construct the Gershgorin disks of $-h^2\frac{L_h+L_h^*}{2}$ and $-h^2\frac{L_h-L_h^*}{2}$. The eigenvalues of a matrix lie on the union of its Gershgorin disks. They are defined as
\begin{equation*}
\mathcal{D}_i = \{z\in\mathbb{C}:|z-a_{jj}|\leq\sum_{j\neq i}^{N}|a_{ij}|\}.
\end{equation*}
The scaled operator is
\begin{equation*}
-h^2L_h = \left(
  \begin{array}{ccccccccc}
     2&-1&                  &                                                     &                                  &                         &                  &                  &                  \\
    -1& 2&    -1            &                                                     &                                  &                         &                  &                  &                  \\
      &  &\ddots            &                                                     &                                  &                         &                  &                  &                  \\
      &  &    -1            &               2                                     &         -1                       &                         &                  &                  &                  \\
      &  &                  &-\frac{2}{1+\gamma}                                   &2\frac{1+\frac{1}{\gamma}}{1+\gamma}&-\frac{2}{\gamma(1+\gamma)}&                  &                  &                  \\
      &  &                  &                                                     &-\frac{1}{\gamma^2}                &\frac{2}{\gamma^2}        &-\frac{1}{\gamma^2}&                  &                  \\
      &  &                  &                                                     &                                  &                         &\ddots            &                  &                  \\
      &  &                  &                                                     &                                  &                         &-\frac{1}{\gamma^2}&\frac{2}{\gamma^2} &-\frac{1}{\gamma^2}\\
      &  &                  &                                                     &                                  &                         &                  &-\frac{1}{\gamma^2}&\frac{2}{\gamma^2} \\
  \end{array}
\right)
\end{equation*}
\edt{We will use the notation $B(c,r)\subset\mathbb{C}$ for the ball centered around $c\in\mathbb{C}$ with radius $r\in\mathbb{R}$.} So the matrix $-h^2\frac{L_h+L_h^*}{2}$ has seven distinct Gershgorin disks
\begin{eqnarray*}
i=1:                    \mathcal{D}_1 &=& B\left(2,1\right) \\
\forall 2\leq i\leq n-2:\mathcal{D}_2 &=& B\left(2,2\right) \\
i=n-1:                  \mathcal{D}_3 &=& \edt{B\left(2,1+\frac{1}{2}\left|\frac{3+\bar{\gamma}}{1+\bar{\gamma}}\right|\right)} \\
i=n:                    \mathcal{D}_4 &=& \edt{B\left(2\Re(\frac{1+\frac{1}{\gamma}}{1+\gamma}),\left(\frac{1}{2}\left|\frac{3+\gamma}{1+\gamma}\right|+\left|\frac{1}{2\bar{\gamma}^2}+\frac{1}{\gamma(1+\gamma)}\right|\right)\right)} \\
j=m-1:                  \mathcal{D}_5 &=& B\left(2\Re(\frac{1}{\gamma^2}),\Re(\frac{1}{\gamma^2})+\left|\frac{1}{2\gamma^2}+\frac{1}{\bar{\gamma}(1+\bar{\gamma})} \right|\right) \\
\forall~2\leq j\leq m-2:\mathcal{D}_6 &=& B\left(2\Re(\frac{1}{\gamma^2}),2\Re(\frac{1}{\gamma^2})\right) \\
j=1:                    \mathcal{D}_7 &=& B\left(2\Re(\frac{1}{\gamma^2}),\Re(\frac{1}{\gamma^2})\right). \\
\end{eqnarray*}
The minimum and maximum of $\bigcup_{1\leq i\leq7}\mathcal{D}_i=\bigcup_{2\leq i\leq6}\mathcal{D}_i$ determine a lower
and upper bound for $h^2\sigma(-\frac{L_h+L_h^*}{2})$. \edt{For our ECS problems, ECS angles up to $\frac{\pi}{6}$ ($0<\theta_\gamma\leq\frac{\pi}{6}$) the minimum is $\Re(\frac{1}{\gamma^2})-\left|\frac{1}{2\gamma^2}+\frac{1}{\bar{\gamma}(1+\bar{\gamma})}\right|<0$ and the maximum is $\max(4,3+\frac{1}{2}\left|\frac{3+\bar{\gamma}}{1+\bar{\gamma}}\right|)$. Note that the lower bound for the real part of the eigenvalues is negative and does not exclude negative values.}\\
The matrix $-h^2\frac{L_h-L_h^*}{2}$ has five distinct Gershgorin disks
\begin{eqnarray*}
i=n-1:                  \mathcal{D}_1 &=& \edt{B\left(0,\frac{1}{2}\left|\frac{1-\bar{\gamma}}{1+\bar{\gamma}}\right|\right)} \\
i=n:                    \mathcal{D}_2 &=& \edt{B\left(2\imath\Im(\frac{1+\frac{1}{\gamma}}{1+\gamma}),\left(\frac{1}{2}\left|\frac{1-\gamma}{1+\gamma}\right|+\left|\frac{1}{2\bar{\gamma}^2}-\frac{1}{\gamma(1+\gamma)}\right|\right)\right)} \\
j=m-1:                  \mathcal{D}_3 &=& B\left(2\imath\Im(\frac{1}{\gamma^2}),\edt{-\Im(\frac{1}{\gamma^2})}+\left|\frac{1}{2\gamma^2}-\frac{1}{\bar{\gamma}(1+\bar{\gamma})} \right|\right) \\
\forall~2\leq j\leq m-2:\mathcal{D}_4 &=& B\left(2\imath\Im(\frac{1}{\gamma^2}),\edt{-2\Im(\frac{1}{\gamma^2}})\right) \\
j=1:                    \mathcal{D}_5 &=& B\left(2\imath\Im(\frac{1}{\gamma^2}),\edt{-2\Im(\frac{1}{\gamma^2}})\right). \\
\end{eqnarray*}
The minimum and maximum of $\bigcup_{1\leq i\leq5}\mathcal{D}_i=\bigcup_{\edt{1}\leq i\leq4}\mathcal{D}_i$ determine a lower
and upper bound for \edt{$-\imath h^2\sigma(-\frac{L_h-L_h^*}{2})$}. \edt{For our ECS problems, ECS angles up to $\frac{\pi}{6}$ ($0<\theta_\gamma\leq\frac{\pi}{6}$) these extrema are \edt{$4\Im(\frac{1}{\gamma^2})$} and \edt{$\frac{1}{2}\left|\frac{1-\bar{\gamma}}{1+\bar{\gamma}}\right|>0$} respectively}.
\end{proof}
%------------------
% fig: gersh
\begin{figure}[h!]
\begin{center}
\includegraphics[width=9cm]{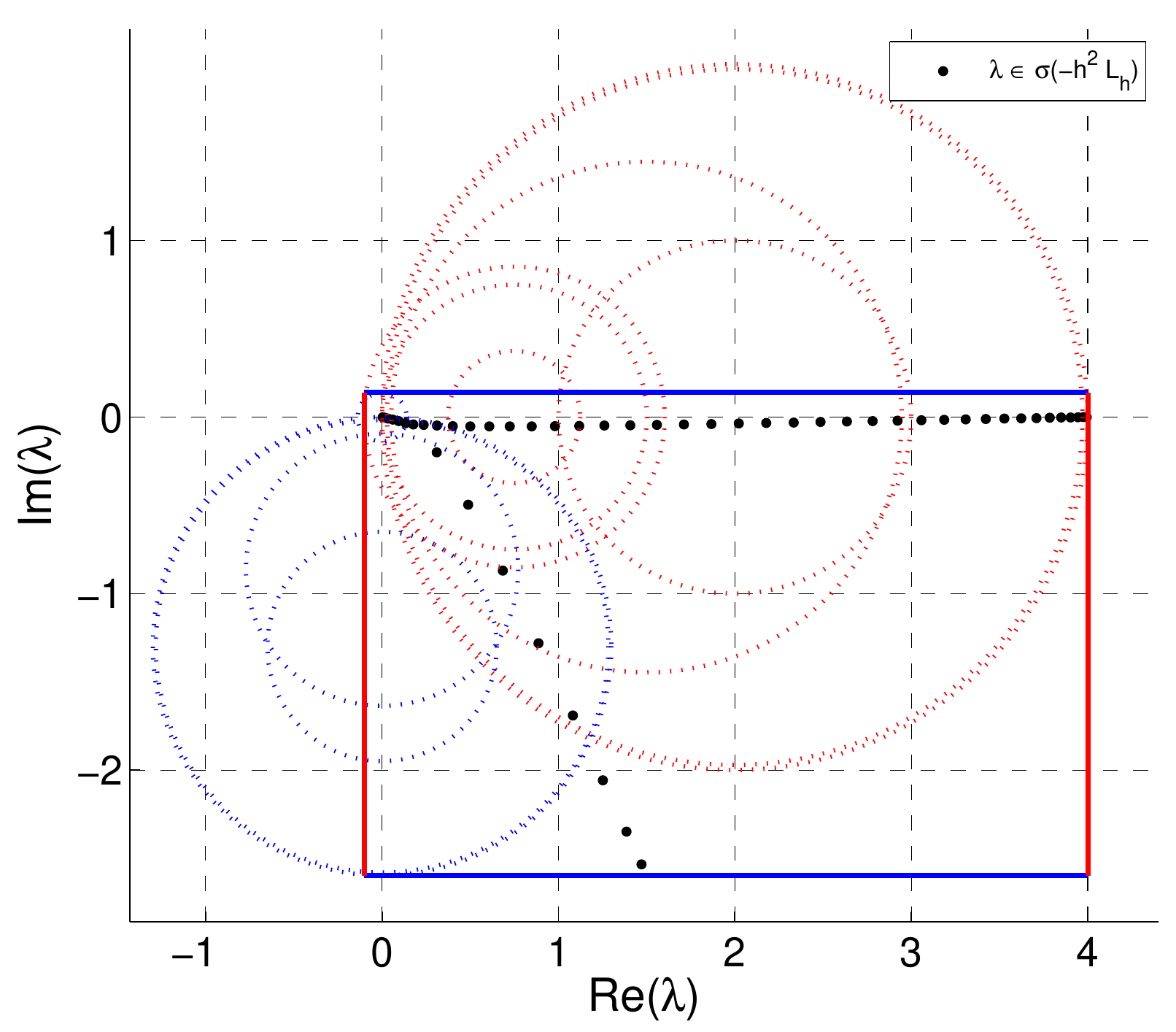}\caption{Spectrum of $-h^2L_h$ ($\bullet$) and the Gershgorin disks of $-h^2/2\left(L_h+L_h^*\right)$ (red dotted circles) and $-h^2/2\left(L_h-L_h^*\right)$ (blue dotted circles). \edt{The spectrum is bounded by the field of values that lies inside the rectangle derived from the Gershgorin disks}. The red circles lead to \edt{the} left and right bound, the blue circles give \edt{the} upper and lower bound. The negative left bound does not assure positive definiteness of the matrix. (color online)}
\label{fig:gersh}
\end{center}
\end{figure}

By considering the Gershgorin disks of only the discretized Laplacian $L_h$ these bounds can be further sharpened.
\edt{Next is a discussion on the exact position of the eigenvalues.}
%------------------
% lemma eigdis
\begin{lem}\label{lem:eigdis}
Consider the ECS grid \eqref{eq:ecsgrid} and the discretization matrix $L_h$ in Equation~\eqref{eq:lapdisecs}. Define $\gamma =
\frac{h_\gamma}{h}$. Then the eigenvalues of $-L_h$ are the solutions of
\begin{equation}\label{eq:eigcond}
F(\lambda) \equiv \frac{\tan(2n p(\lambda))}{\tan(2m q(\lambda))}+\frac{\cos(p(\lambda))}{\cos(q(\lambda))} = 0,
\end{equation}
with $p(\lambda)=\frac{1}{2}\arccos(1-\frac{\lambda}{2}h^2)$, $q(\lambda)=\frac{1}{2}\arccos(1-\frac{\lambda}{2}\gamma^2
h^2)$.
\end{lem}
\begin{proof}
Note that the eigenvalues of the Helmholtz problem fulfill the same condition \eqref{eq:eigcond}, with
$\lambda_{H_h} \equiv \lambda_{L_h}-k^2$. We prove the Laplace case.\\
We write the grid points slightly different \edt{than} before by numbering the grid points $(x_j)_{1\leq n}$ in $[0,1]$ from left to right, and $(y_j)_{1\leq j\leq m}$ in $[1,R_z]$ form right to left, so that the turning point $x_n=y_m=1$. In other words we consider the ECS grid as two joint grids  $(x_j)_{1\leq j\leq n}$ and $(y_j)_{1\leq j\leq m}$.
Consequently the two Dirichlet boundaries are $x_0$ and $y_0$. The Shortley-Weller finite differences formula reduces to regular
second order central differences in every grid point, except for the turning point $x_n=y_m=1$ with left grid distance
$h$ and right grid distance $\gamma h$. We look for an eigenvector $v$ on the real grid and an eigenvector $w$ on the
complex contour with the same eigenvalue $\lambda$ and $v_n=w_m$ in the turning point. We get the recurrence relations
\begin{equation*}
\left\{
  \begin{array}{ll}
    -\frac{1}{h^2}\left(v_{j-1}-2v_{j}+v_{j+1}\right) = \lambda v_j, & 0< j< n; \\
    -\frac{2}{h^2(1+\gamma)}\left(v_{n-1}-\left(1+\frac{1}{\gamma}\right)v_n+\frac{1}{\gamma}w_{m-1}\right) = \lambda v_n, & j = n; \\
    -\frac{1}{\gamma^2h^2}\left(w_{j-1}-2w_{j}+w_{j+1}\right) = \lambda w_j, & 0< j< m.
  \end{array}
\right.
\end{equation*}
The first and the last are Chebyshev recurrence relations with general solutions
\begin{equation*}
\left\{
  \begin{array}{ll}
    v_j = c_1 V_j(s_\lambda)+c_2T_j(s_\lambda), & 0< j< n; \\
    w_j = d_1 V_j(t_\lambda)+d_2T_j(t_\lambda), & 0< j< m.
  \end{array}
\right.
\end{equation*}
with $s_\lambda = 1-\frac{\lambda}{2}h^2$, $t_\lambda = 1-\frac{\lambda}{2}(\gamma h)^2$ and
$V_j(s_\lambda)=\sqrt{1-s_\lambda}U_{j-1}(s_\lambda)$. $T_j$ and $U_{j}$ are the $j$-th Chebyshev polynomials of the first
and the second kind respectively. We apply the boundary conditions $u_0=0$ and $v_0=0$ and find
\begin{equation*}
\left\{
  \begin{array}{ll}
    v_j = c_1 V_j(s_\lambda), & 0< j< n; \\
    w_j = d_1 V_j(t_\lambda), & 0< j< m.
  \end{array}
\right.
\end{equation*}
The matching condition returns
\begin{eqnarray*}
&& v_n = w_m \\
&\Leftrightarrow& c_1 V_n(s_\lambda) = d_1 V_m(t_\lambda) \\
&\Leftrightarrow& d_1 = c_1 \frac{V_n(s_\lambda)}{V_m(t_\lambda)} \\
\end{eqnarray*}
where we assumed $v_n\neq0$ so $V_n(s_\lambda)\neq0\neq V_m(t_\lambda)$. The remaining constant $c_1$ determines the
norm of the eigenvector. So we can assume $c_1=1$. Now the only free parameter left is the eigenvalue $\lambda$ that is determined by the second
recurrence relation, the discretization scheme in the turning point.
\begin{eqnarray*}
&& -\frac{2}{h^2(1+\gamma)}\left(V_{n-1}(s_\lambda)-\left(1+\frac{1}{\gamma}\right)V_n(s_\lambda)+\frac{1}{\gamma}\frac{V_n(s_\lambda)}{V_m(t_\lambda)}V_{m-1}(t_\lambda)\right) = \lambda V_n(s_\lambda)\\
&\Leftrightarrow& -\frac{2}{h^2(1+\gamma)}\left(\frac{V_{n-1}(s_\lambda)}{V_n(s_\lambda)}-\left(1+\frac{1}{\gamma}\right)+\frac{1}{\gamma}\frac{V_{m-1}(t_\lambda)}{V_{m}(t_\lambda)}\right) = \lambda\\
%&\Leftrightarrow& \frac{V_{m-1}(t_\lambda)}{V_m(t_\lambda)} +\gamma\frac{V_{n-1}(s_\lambda)}{V_n(s_\lambda)} = -\lambda\frac{h^2\gamma(1+\gamma)}{2}+\gamma+1\\
%&\Leftrightarrow& \frac{V_{m-1}(t_\lambda)}{V_m(t_\lambda)} +\gamma\frac{V_{n-1}(s_\lambda)}{V_n(s_\lambda)} = \gamma(1-\lambda\frac{h^2}{2} +(1-\lambda\frac{\gamma^2h^2}{2})\\
&\Leftrightarrow& \frac{V_{m-1}(t_\lambda)}{V_m(t_\lambda)} +\gamma\frac{V_{n-1}(s_\lambda)}{V_n(s_\lambda)} = \gamma s_\lambda +t_\lambda\\
&\Leftrightarrow& \frac{\sin((m-1)\arccos(t_\lambda))}{\sin(m\arccos(t_\lambda))} +\gamma\frac{\sin((n-1)\arccos(s_\lambda))}{\sin(n\arccos(s_\lambda))} = \gamma s_\lambda +t_\lambda\\
&\Leftrightarrow& t_\lambda-\cot(2mq_\lambda)\sqrt{1-t_\lambda^2} +\gamma\left(s_\lambda-\cot(2np_\lambda)\sqrt{1-s_\lambda^2}\right) = \gamma s_\lambda +t_\lambda\\
&\Leftrightarrow& -\cot(2mq_\lambda)\sqrt{1-t_\lambda^2} -\gamma\cot(2np_\lambda)\sqrt{1-s_\lambda^2} = 0 \\
&\Leftrightarrow& \frac{\tan(2np_\lambda)}{\tan(2mq_\lambda)} = -\gamma\sqrt{\frac{1-s_\lambda^2}{1-t_\lambda^2}} \\
&\Leftrightarrow& \frac{\tan(2np_\lambda)}{\tan(2mq_\lambda)} +\sqrt{\frac{1+s_\lambda}{1+t_\lambda}} = 0 \\
\end{eqnarray*}
\edt{We introduced the shorthands $p_\lambda =\arccos(t)$ and $q_\lambda =\arccos(s)$, substituted $V_j(x)=\sin(j \arccos(x))$ and excluded the trivial cases $\cot(2np_\lambda)=0$ and $t_\lambda^2=1$.}
\end{proof}
%----------------
We can now solve the eigenvalue problem numerically by applying e.g.\ Newton's method on the function \eqref{eq:eigcond}. There are eigenvalues to be found along the complex line $\rho e^{-2i\theta_\gamma}$ with $\rho\in\mathbb{R}$, and close to $4/h^2$ and $4/\gamma^2 h^2$.
%----------------
\begin{lem}\label{lem:pitchfork}
Let $-L_h$ be the negative discretized Laplacian in Equation~\eqref{eq:lapdisecs} with eigenvalues $\lambda\in\sigma(-L_h)$. Then three typical regions of the eigenvalues can be identified in the spectrum:
\begin{align*}
\mbox{For } |\lambda-\frac{4}{h^2}|\ll 1: \lambda\approx4 n^2\sin(l \frac{\pi}{2n}) \mbox{ with } l\lesssim n\\
\mbox{For } |\lambda-\frac{4}{h^2\gamma^2}|\ll 1: \lambda\approx4 \gamma^2 m^2\sin(l \frac{\pi}{2m})  \mbox{ with } l\lesssim m \\
\mbox{For } |\lambda|\ll 1: \lambda\approx \left(\frac{l\pi}{R_z}\right)^2  \mbox{ with } l\gtrsim 1\\
\end{align*}
These approximations are the largest eigenvalues of the discretized Laplacian restricted to the real domain $[0,1]$, the complex contour $[1,R_z]$, and the smallest eigenvalues of the discretized Laplacian on the complex line $[0,R_z]$, respectively.
\end{lem}
\begin{proof}
We take the scaled operator $-h^2 L_h$. The eigenvalues are the roots of the function
\begin{align}
F(\mu) &= \sin\left(2np(\mu)\right)\cos\left(2mq(\mu)\right)\cos\left(q(\mu)\right) \\ \nonumber
 &+\cos\left(2np(\mu)\right)\sin\left(2mq(\mu)\right)\cos\left(p(\mu)\right) \label{eq:F}
\end{align}
with $p(\mu)=\frac{1}{2}\arccos(1-\frac{\mu}{2})$, $q(\mu)=\frac{1}{2}\arccos(1-\frac{\mu}{2}\gamma^2)$.
Using Taylor series we get
\begin{align*}
 p(\mu) &= \frac{\pi}{2}-\frac{1}{2}\sqrt{4-\mu} +\mathcal{O}(|\edt{4-\mu}|^{3/2})
\end{align*}
for $\mu\approx 4$, so $F(\mu)$ can be approximated by
\begin{align*}
F(\mu) &\approx \sin\left(n\left(\pi-\sqrt{4-\mu}\right)\right)\cos\left(2mq(\mu)\right)\cos\left(q(\mu)\right)\\
    &+\cos\left(n\left(\pi-\sqrt{4-\mu}\right)\right)\sin\left(2mq(\mu)\right)\sin\left(\frac{1}{2}\sqrt{4-\mu}\right)
\end{align*}
for $|\mu-4| \ll 1$. This can be simplified even more to
\begin{align}
F(\mu) &\approx \sin\left(n\left(\pi-\sqrt{4-\mu}\right)\right)\cos\left(2mq(\mu)\right)\cos\left(q(\mu)\right)\\ \nonumber
    &+\cos\left(n\left(\pi-\sqrt{4-\mu}\right)\right)\sin\left(2mq(\mu)\right)\left(\frac{1}{2}\sqrt{4-\mu}\right),\label{eq:Fsimp1}
\end{align}
where we used the series
\begin{align*}
\sin\left(\frac{1}{2}\sqrt{4-\mu}\right) &= \frac{1}{2}\sqrt{4-\mu} +\mathcal{O}(|4-\mu|^{3/2})
\end{align*}
with $|\mu-4| \ll 1$. Define $\mu_l=4-\frac{\pi^2}{n^2}\left(l-n\right)^2$, then
\begin{align*}
F(\mu_l) &\approx 0 \pm\sin\left(2mq(\mu_l)\right)\left(\frac{1}{2}\sqrt{4-\mu_l}\right) \\
    & \approx 0
\end{align*}
for $l\approx n$.
The eigenvalues of the discretized Laplacian with Dirichlet boundary conditions, defined on the real domain $[0,1]$, are $\lambda_l = 4\sin^2(\frac{l\pi}{n})$ with $1\leq l\leq n-1$. For $l\approx n$ we have
\begin{align*}
\lambda_l &= 4-\frac{\pi^2}{n^2}\left(l-n\right)^2 +\mathcal{O}((l-n)^3) \\
    &= \mu_l +\mathcal{O}((l-n)^3).
\end{align*}
So we found that the eigenvalues of $-h^2 L_h$ in the neighborhood of $4$ can be approximated by the eigenvalues of the Laplacian on the real part of the domain. In the same way we can show that the eigenvalues in the neighborhood of $\frac{4}{\gamma^2}$ can be approximated by the eigenvalues of the Laplacian defined on the complex contour, by approximating
\edt{\begin{align*}
q(\mu) &= \frac{\pi}{2}-\frac{1}{2}\sqrt{4-\gamma^2\mu} +\mathcal{O}(|\edt{4-\gamma^2\mu}|^{3/2}).
\end{align*}}
Now we are looking for the smallest eigenvalues of $-h^2 L_h$. Using Taylor series we get
\begin{align*}
p(\mu) &= \frac{\sqrt{\mu}}{2} +\mathcal{O}(|\mu|^{3/2}) \\
q(\mu) &= \gamma\frac{\sqrt{\mu}}{2} +\mathcal{O}(|\mu|^{3/2})
\end{align*}
so $F(\mu)$ can be approximated by
\begin{align*}
F(\mu) &\approx \sin\left(n\sqrt{\mu}\right)\cos\left(m \gamma\sqrt{\mu}\right)\cos\left(\frac{\sqrt{\mu}\gamma}{2}\right)\\
    & +\cos\left(n\sqrt{\mu}\right)\sin\left(m\gamma\sqrt{\mu}\right)\cos\left(\frac{\sqrt{\mu}}{2}\right)
\end{align*}
for $\mu \ll 1$. We write $\gamma = 1+\imath\varepsilon$ with $0<\varepsilon<1$. This is true for \edt{ECS} with an angle $0<\theta_\gamma<\frac{\pi}{4}$. Then the function can be simplified even more to
\begin{align}
F(\mu) &\approx \sin\left(n\sqrt{\mu}+m\gamma\sqrt{\mu}\right)\cos(\frac{\sqrt{\mu}}{2})
\end{align}\label{eq:Fsimp3}
where we used the series
\begin{align*}
\cos\left(\gamma\frac{\sqrt{\mu}}{2}\right) &= \cos\left(\frac{\sqrt{\mu}}{2}\right) +\mathcal{O}(|\mu\varepsilon|)
\end{align*}
The eigenvalues $\lambda_l = \left(\frac{l \edt{\pi}}{n + m\gamma}\right)^2$ of the scaled continuous operator $-h^2 L$, with $l\in\mathbb{N}$, are roots of the simplified function \eqref{eq:Fsimp3}. So the smallest eigenvalues of \eqref{eq:lapdisecs} can be approximated by $\lambda_l\ll1$.
\end{proof}
%------------------
For the Laplace problem \eqref{eq:lapdisecs} the spectrum has a typical pitchfork shape. There is a clear complex
branch associated to eigenvectors located on the complex contour, and a branch closer to the real axis that corresponds
to eigenvectors located on the real domain. The smallest eigenvalues in the tail of the pitchfork belong to the
smoothest eigenvectors spread over the entire ECS domain. They lie close to the smallest eigenvalues of the continuous
ECS operator $-L$ (Figure \ref{fig:pitchfork}). Indeed, define the complex mesh width $h_\alpha = R_z/(n+m+1)$, belonging to a straight complex grid connecting $0$ and $R_z$, and $\alpha = h_\alpha/h$.
Then we conjecture for the discretized Laplacian $-L_h$ in \eqref{eq:lapdisecs}
$$\sigma(-L_h)\subset (\mathfrak{S}\cup\mathfrak{T})$$
where $\mathfrak{S}$ is a strip around the complex line $4\rho/(\alpha h)^2$  $(0\leq\rho\leq\rho_0<1)$ and
$\mathfrak{T}$ is the interior of the triangle $\widehat{\mu_1\mu_2\mu_3}\subset\mathbb{C}$, with
$\mu_1 = 4\rho_0 /(\alpha h)^2$, $\mu_2 = 4/h^2$ and $\mu_3 = 4/(\gamma h)^2$. We liberally use the terms, \emph{pitchfork} and \emph{tail of the pitchfork} to represent the triangular region $\mathfrak{T}$ and the line segment $\mathfrak{S}$ respectively. \edt{For the Helmholtz operator with a constant wave number $k$ the pitchfork is shifted in the negative real direction over a distance $k^2$.}
% fig: pitchfork
\begin{figure}[h!]
\begin{center}
\includegraphics[width=9cm]{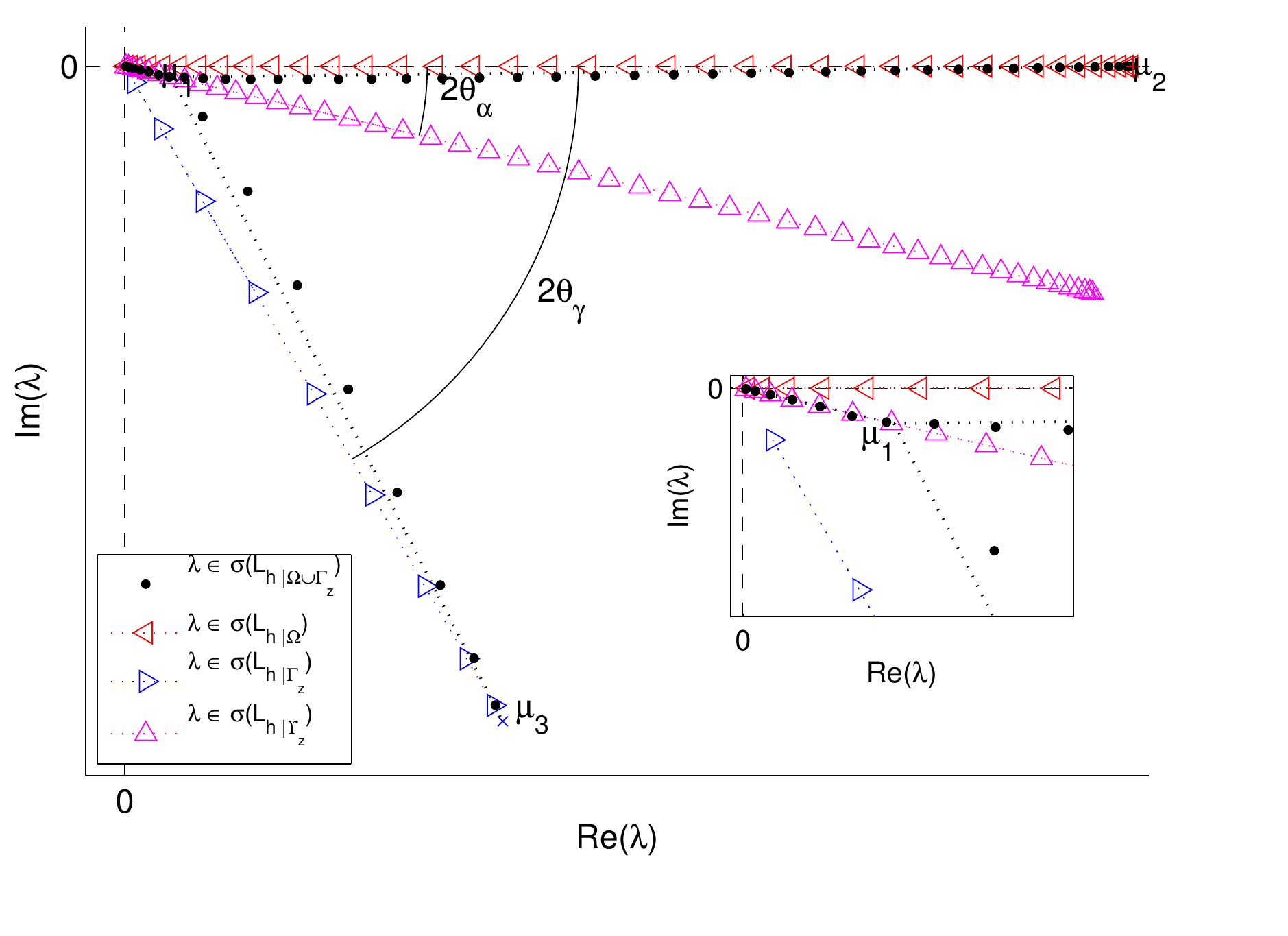}\caption{An illustration of the result in Lemma~\ref{lem:pitchfork}. The eigenvalues of the ECS Laplacian discretization matrix ($\bullet$) lie along a pitchfork shape figure, close to the eigenvalues of the same \edt{Laplace} problem restricted to the interior real domain ($\triangleleft$) and the complex contour ($\triangleright$) respectively. In the detail view of the area around the origin we observe that the smallest eigenvalues practically agree with the smallest eigenvalues of the Laplace problem defined on the complex line $[0,R_z]$ ($\triangle$). The result is a pitchfork shape: the smallest eigenvalues are aligned until they split up into two branches in the point $\mu_1 = 4\rho_0/(\alpha h)^2$, with limiting points $\mu_2 = 4/h^2$ ($+$) and $\mu_3 = 4/(\gamma h)^2$ ($\times$). (color online)}
\label{fig:pitchfork}
\end{center}
\end{figure}

\edt{
\begin{rem}[Spectrum of smoother ECS transformations]
In our analysis we have focused on problems on a domain that is complex stretched by the ECS transformation in \eqref{eq:lapctuecs}. This leads, in 1D, to a tridiagonal matrix which is constant in the interior and constant in the boundary layer. Both regions are connected by a single condition which is the finite difference approximation of the equation at the turning point. For other complex stretching transformations like quadratic or polynomial scaling (see e.g.\ Figure~\ref{fig:smoothECS}) the discretization matrix is no longer constant and it is much harder to derive theoretical results for the eigenvalues. Numerical experiments, however, show a very similar eigenvalue spectrum with a pitchfork. Again, here are some numerical eigenvalues, corresponding to smooth modes, that approximated the analytical result, $\left(j\pi /R_z\right)^2$. At the pitchfork, the spectrum breaks again into two branches. One branch belongs to eigenmodes that are mainly located in the boundary layer and these modes have eigenvalues with a large imaginary part. The other branch corresponds to states that are located on the real part of the grid and the eigenvalues will lie close to the real axis.
\end{rem}
}
\begin{figure}
\begin{center}
\includegraphics[width=13cm,height=6.5cm]{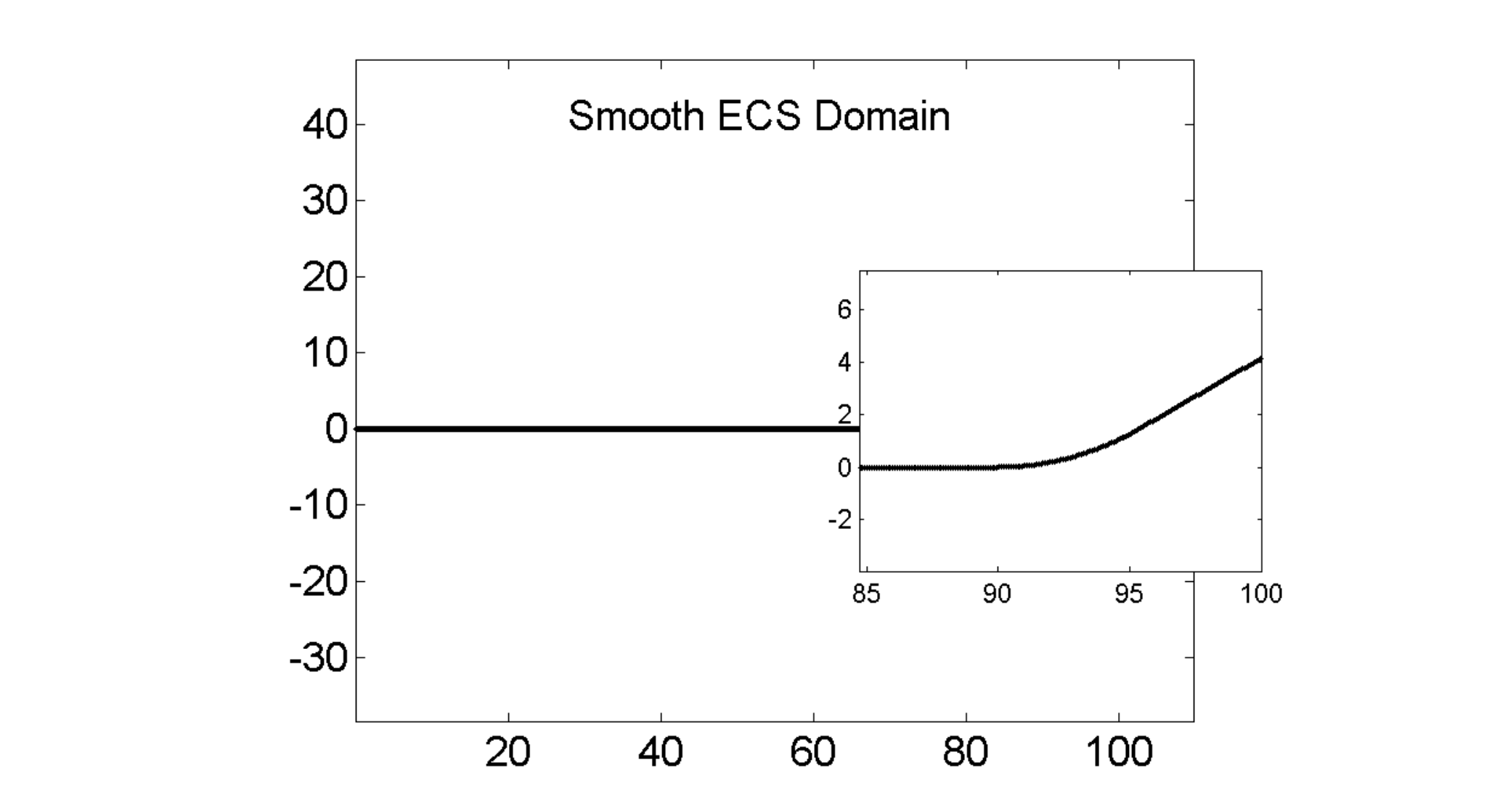}
	\caption{A one-dimensional domain with smoother ECS transition in order to preserve the order of discretization. The turning point is magnified.}\label{fig:smoothECS}
\end{center}
\end{figure}

\section{The Model problems} \label{sec:MP}
%--------------------------------------
% experiments: model problems
%--------------------------------------
The actual discrete problems that we solve in the section with numerical experiments, Section \ref{sec:experiments}, are derived from three model problems \edt{that are representative for break-up problems as they appear in physical systems.} The discrete formulation is constructed both with the first order Sommerfeld radiation boundary conditions as well as with the ECS \edt{layers}. Therefore, the boundary conditions are not part of the nomenclature. E.g., MP1 refers to Model Problem 1 and does not take into account the boundary conditions, which would be explicitly mentioned. Likewise MP2 and MP3. Collectively, these model problems are given by the Helmholtz equation,
\begin{align} \label{eq:Helgen}
 - \{\Delta + \phi(x,y)\}u(x,y) = \chi(x,y); \qquad (x,y) \in (0,r)^2
\end{align}
and are distinguished by the concrete form of the \edt{space dependent} wave number $\phi$, the right hand side $\chi$, and the domain size $r$.

For a Helmholtz equation on a unit square domain with a wave number $\phi(x,y) = k\edt{^2}$, an accuracy condition that guards against phase errors \emph{polluting} the computations \cite{BGT85,IB95}, requires bounding \edt{$k^3 h^2$} by a small constant. A similar \edt{but less} strict constraint that ensures using at least $10$ points per wavelength of the solution translates to $kh<0.625$ \cite{EOV06}. \edt{Since we plan to observe solely iterative behavior, we will stick to the latter relaxed condition in our experiments.}

\subsection{Model Problem 1 (MP1)}
MP1 is the same model problem that forms the basis for the results that appeared in \cite{EVO06}. It is characterized by a point source in the center of the domain embodied by a Dirac delta right hand side. In the discrete version of the problem, the right hand side is non-zero ($=1$) for only one computational node in the scheme, which lies at the center of the domain. The wave number is constant in MP1, and the domain is a unit square.
\begin{align} \label{eq:MP1}
\text{Specification of MP1:}\\
\phi(x,y) &= k^2, & k \in \mathbb{R}  \qquad(\text{constant wave number}) \notag \\
\chi(x,y) &= \delta(x,y) & \delta(x,y) \qquad(\text{Dirac's delta function}) \notag \\
r &= 1. \notag
\end{align}
A plot of MP1 with both type of boundary treatments is given in Figure \ref{fig:MP1}.
% fig: plotMP1 and plotMP1SOM
\begin{figure}[!h]
 \subfigure[MP1 with \edt{ECS layers} ($\theta_\gamma = \frac{\pi}{6}$)]{\includegraphics[width=8cm]{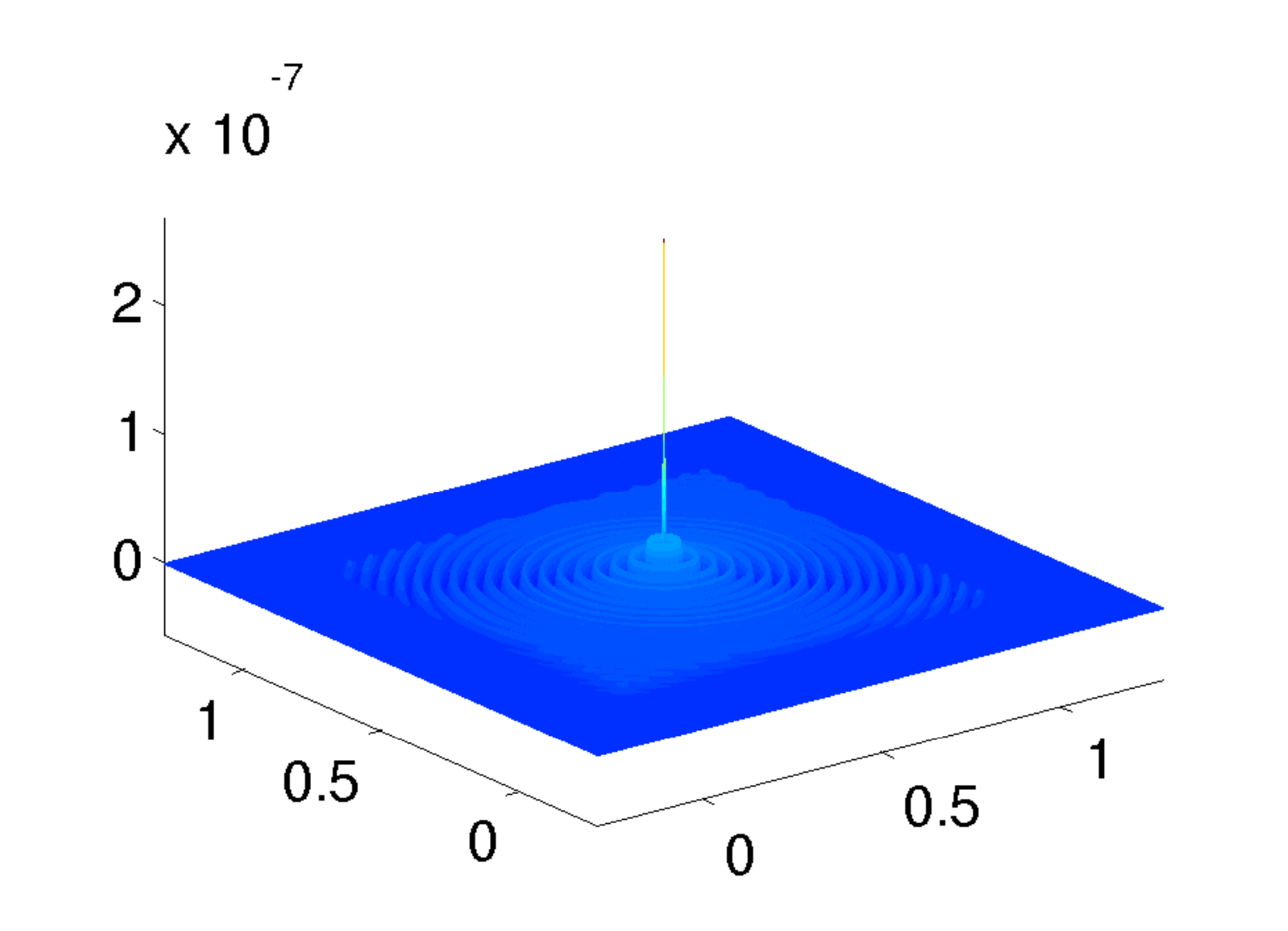}}
 \subfigure[MP1 with Sommerfeld BC]{\includegraphics[width=8cm]{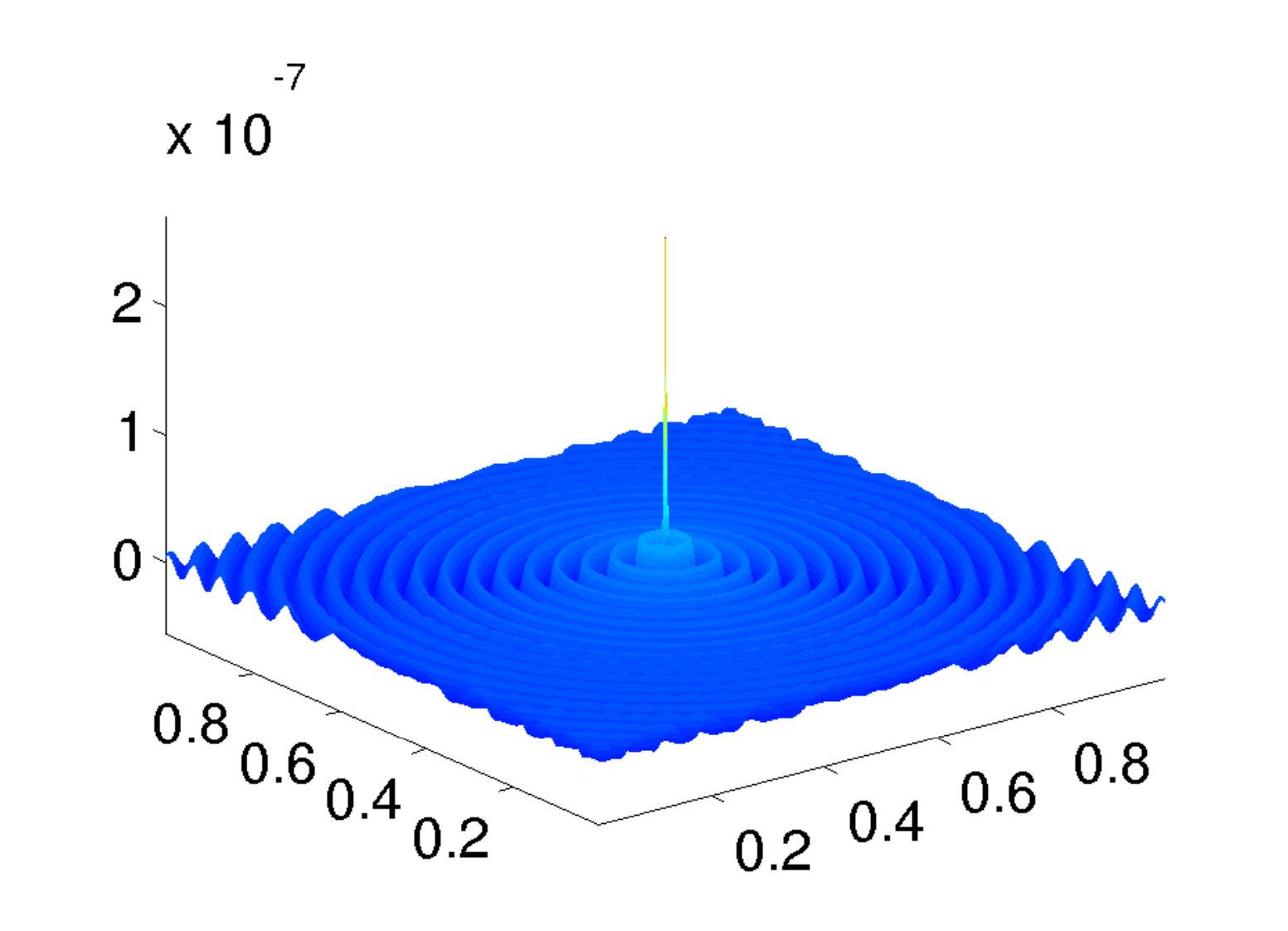}}
\caption{The real part of the solution of MP1 with $k=160$, discretized on a grid having $256$ interior points per dimension. This suffices to meet the accuracy condition. (color online)}
\label{fig:MP1}
\end{figure}

\subsection{Model Problems 2 and 3, (MP2, MP3)}
MP2 and MP3 are Helmholtz model problems with strongly varying wave numbers, and therefore pose a tough benchmark for an iterative approach. These problems originate from large scale Helmholtz problems that appear in the simulation of Schr\"odinger's equation for single and multiple ionization of atoms and molecules \cite{VMRM05}. The dynamics of two \edt{positively charged} electrons $\mathbf{r}_1$ and $\mathbf{r}_2$ in the field of \edt{the  negatively charged nuclei} need to be modeled and solved. \edt{Usually the electrons move, due to their mass difference, much faster than the nuclei and the latter are usually taken fixed in space.} This leads to a Schr\"odinger equation with a six-dimensional wave function $\psi(\mathbf{r}_1,\mathbf{r}_2)$\edt{, which is often} expanded in spherical coordinates about the center of the molecule \edt{with the $z$ axis along the axis of the molecule. The expansion is}
\begin{equation*}
\sum_{l_1m_1,l_2m_2} \psi_{l_1m_1,l_2m_2}(\rho_1,\rho_2) Y_{l_1m_1}(\edt{\Omega_1})Y_{l_2m_2}(\edt{\Omega_2}),
\end{equation*}
where $(\rho_1,\edt{\Omega_1})$ and $(\rho_2,\edt{\Omega_2})$ are the spherical coordinates of the first and the second electron \edt{and $\Omega_i$ denotes the two angles in the spherical coordinates}. This expansion leads to a very large number of coupled 2D problems, where $\psi_{l_1m_1,l_2m_2}(\rho_1,\rho_2)$ is the solution of the 2D problem for particular \edt{integer values of} $l_1$, $m_1$, $l_2$ and $m_2$. It describes the wave as a function of the distances $\rho_1$ and $\rho_2$ of both electrons \edt{to the center of the molecule}. The coupled \edt{equation} has a block structure and the differential operators only appear in the diagonal blocks \edt{since $Y_{lm}(\Omega)$ are eigenfunctions of the angular part of the Laplacian operator in spherical coordinates.}

In the work \cite{VMRM05}, the resulting linear systems for a molecule are solved iteratively. The problem was preconditioned by
inverting the diagonal blocks with the direct sparse solver SuperLU, which is based on the left-looking supernodal method \cite{D06} and was employed on a massively parallel computer. The current work aims to replace this direct solver with an iterative alternate based on preconditioning. The main motivation for studying the numerical properties of the solver for the model problems is that the approach in \cite{VMRM05} can not be used for solving systems with three or more particles. The diagonal blocks that need to be inverted in this extended case are at least 3D problems, and the resulting storage and computational complexity grows out of reach for the current computational infrastructures.

\begin{align} \label{eq:MP2}
\text{Specification of MP2:}\\
\phi(x,y) &= \nu\left(\frac{1}{e^{x^2}} + \frac{1}{e^{y^2}}\right) + k^2, & 0<k<5, \quad 0<\nu<10 \notag \\
\chi(x,y) &= \frac{1}{e^{x^2+y^2}} \notag \\
r &= 50. \notag
\end{align}

\begin{align} \label{eq:MP3}
\text{Specification of MP3:}\\
\phi(x,y) &= \frac{1}{x} + \frac{1}{y} + k^2, & \hspace{1.5cm}0<k<5 \hspace{1.5cm}\notag \\
\chi(x,y) &= \frac{1}{e^{x^2+y^2}} \notag \\
50<& r<200 \notag
\end{align}
These model problems are representative of the 2D problem that appear when $l_1=0$ and $l_2=0$. The coordinates $x$ and $y$ should be interpreted as radial variables $\rho_1$ and $\rho_2$.

Homogeneous Dirichlet boundary conditions stay fixed at the south and the west edges of the domain where $\rho_1$ and $\rho_2$ are
zero. On the east and the north edges where $\rho_1$ or $\rho_2$ are large, absorbing boundary conditions have to be used. We therefore toggle between the first order Sommerfeld BC and the \edt{ECS layers} on these two edges, and provide numerical results with both. \edt{For quality calculations that reproduce the physical experiments, higher order absorbing boundary conditions need to be used, but for the purposes of the current paper these low order boundary conditions are sufficient.} The boundary conditions that MP2 and MP3 employ are given by Equation~\eqref{eq:BC}.
\begin{align} \label{eq:BC}
&u(0,y) = u(x,0) = 0 \quad\text{Homogeneous Dirichlet BC, south/west edges}\\
&\edt{\begin{cases}
\frac{\partial u}{\partial \hat{n}} & = -\imath k u \quad \text{Sommerfeld BC, east/north edges}  \\
\text{or} &\\
u &= 0  \quad\text{ECS layers, east/north edges}
\end{cases}}
\end{align}

For $\nu = 7$ and $k=2$, a plot of the solution of MP2 appears in Figure \ref{fig:MP2}. In this particular case, the minimum grid size required for an acceptable resolution of the solution is $341^2$, closest to which the most convenient practical grid size is $384^2$ from a multigrid perspective. Later in Section \ref{subsec:numex}, we describe how we evaluate the minimum grid size for MP2 and MP3. The solution has evanescent waves for values of $\nu > 2.73$, which are damped exponentially on these
edges. \edt{These evanescent waves correspond to single ionization break-up reactions \cite{RBIM99}.}

In MP3, $\phi(x,y)$ has a singularity at the origin. Unlike MP2, the solution of MP3 always has evanescent waves at the south/west edges, regardless of the choice of the parameters $r$ and $k$. For $r=90$, and $k=2$, the minimum interior grid size required is around $1024^2$ and the solution is depicted in Figure \ref{fig:MP3}.

% fig: plotMP2 and plotMP2SOM
\begin{figure}[!h]
\subfigure[MP2 with \edt{ECS layers} ($\theta_\gamma = \frac{\pi}{6}$)]{\includegraphics[width=8cm]{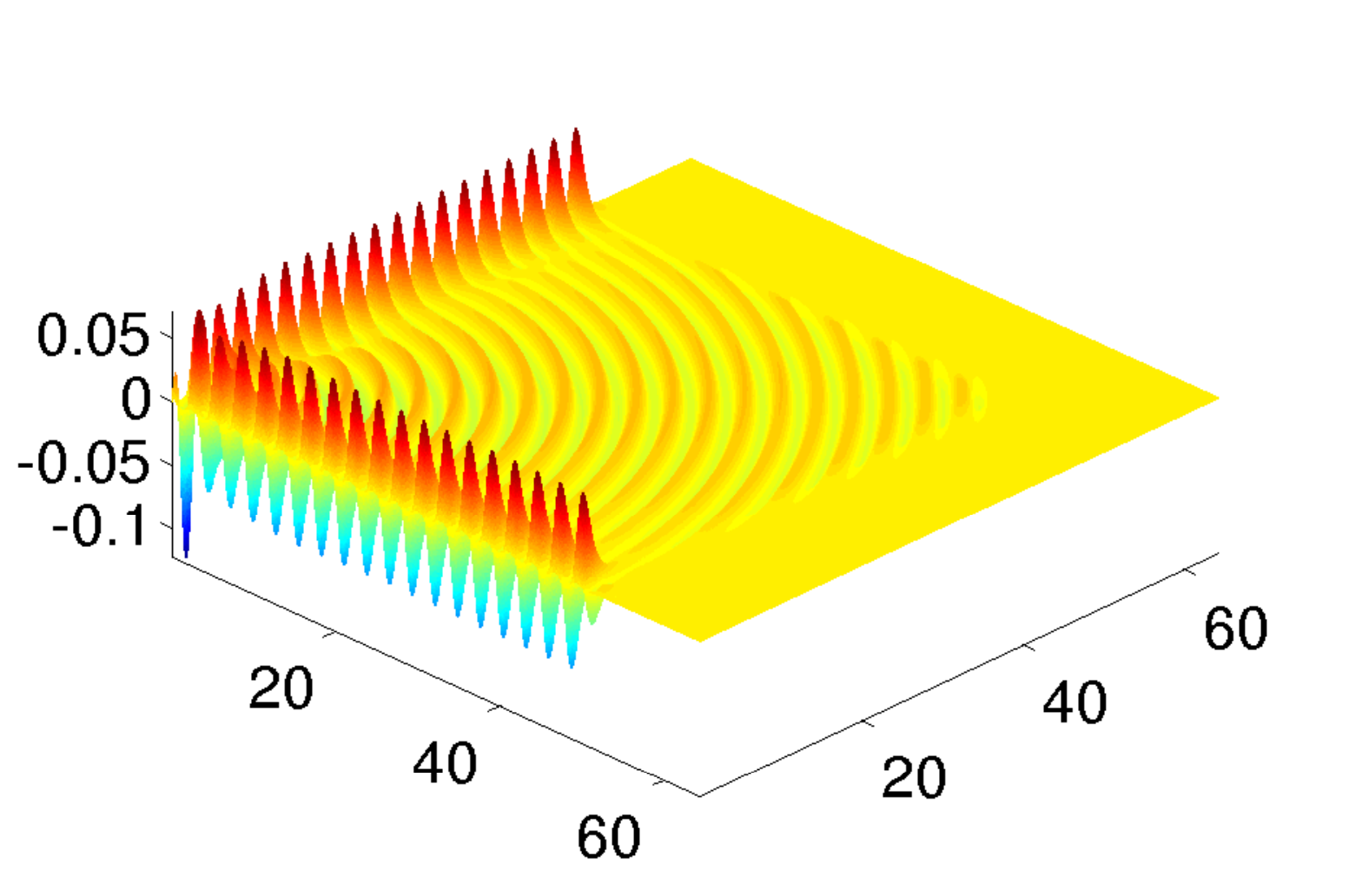}}
\subfigure[MP2 with Sommerfeld BC]{\includegraphics[width=8cm]{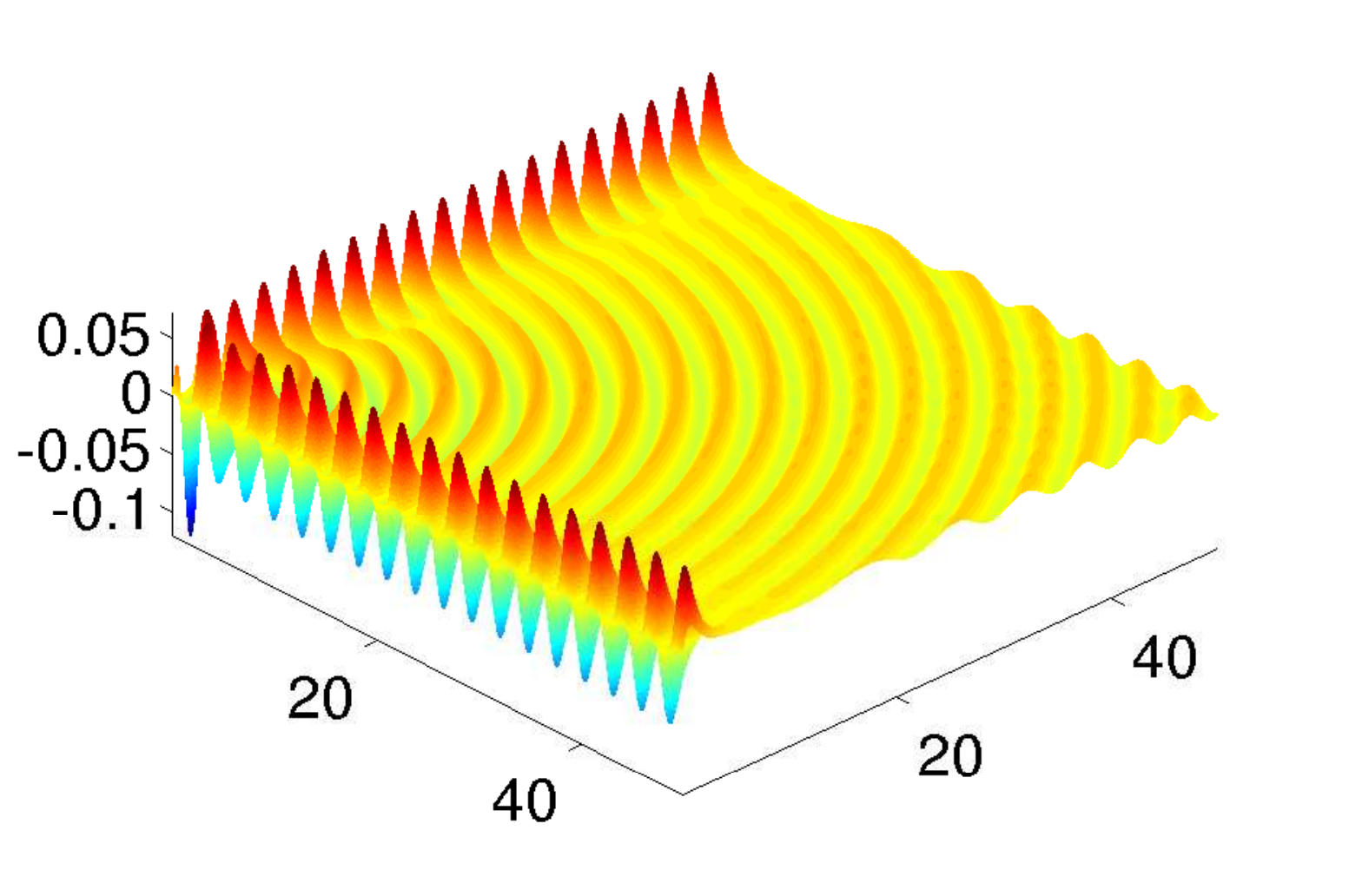}}
\caption{The real part of the solution of MP2, with $\nu = 7$ and $k=2$. For values of $\nu>2.73$, evanescent waves form near Dirichlet edges. (color online)} \label{fig:MP2}
\end{figure}

% fig: MP3ECS and MP3SOM
\begin{figure}
\subfigure[MP3 with \edt{ECS layers} ($\theta_\gamma = \frac{\pi}{6}$)]{\includegraphics[width=8cm]{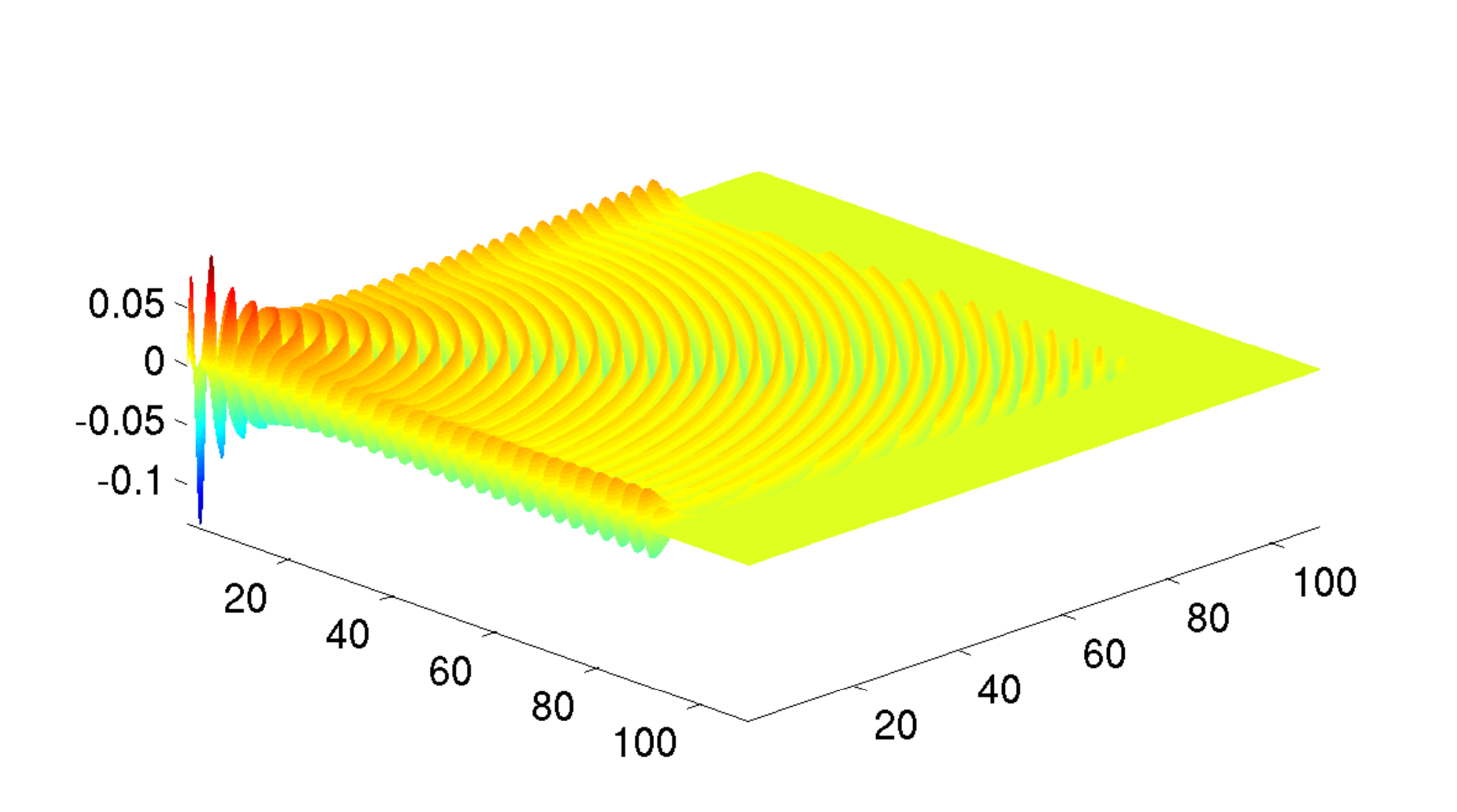}}
\subfigure[MP3 with Sommerfeld BC]{\includegraphics[width=8cm]{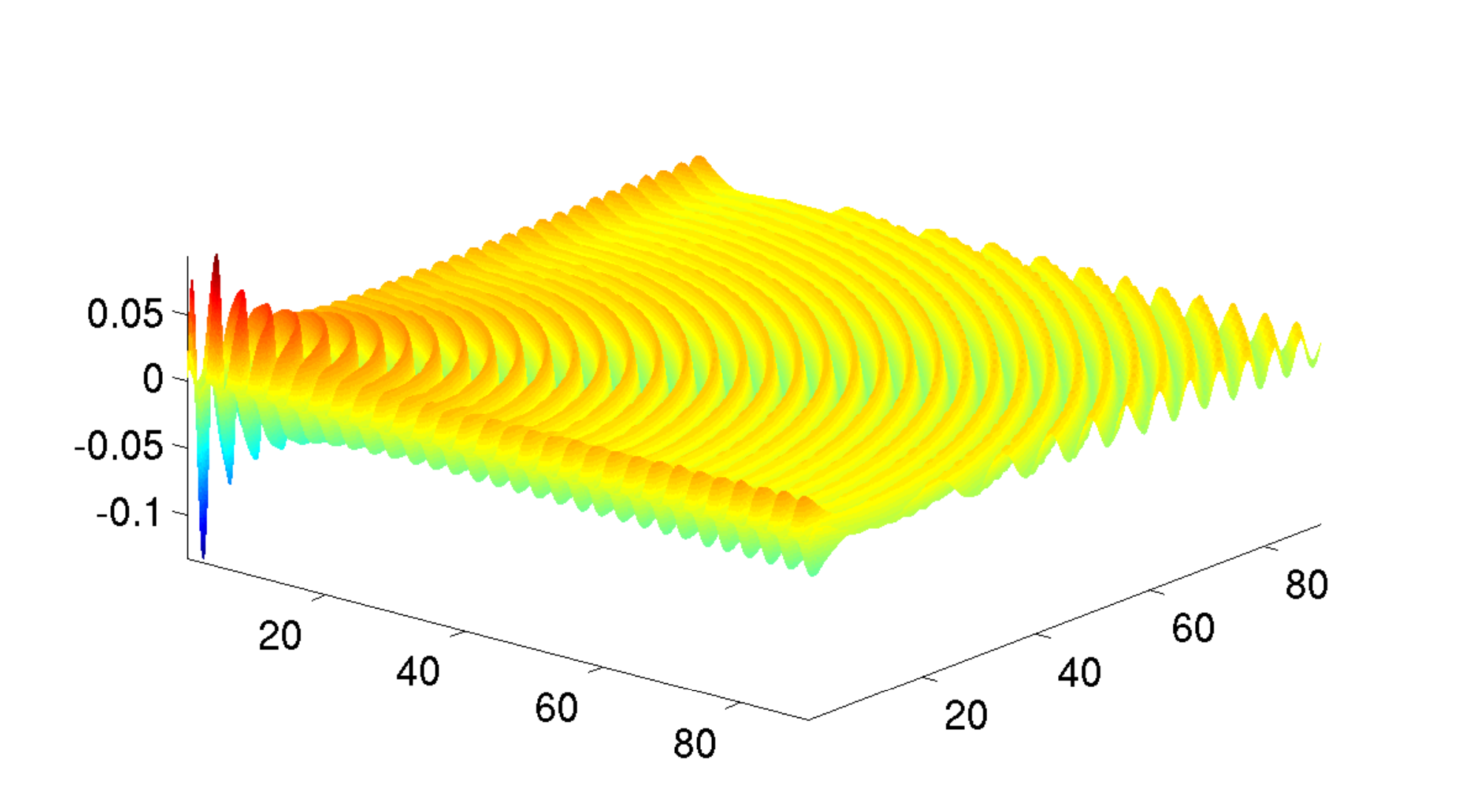}}
\caption{The \edt{real part of the} solution of MP3, with $r=90$ and $k=2$. Evanescent waves are visible near Dirichlet boundaries. The
wavelength is quite small and implies a huge grid size to comply with accuracy requirements. (color online)} \label{fig:MP3}
\end{figure}

This completes the description of the model problems, which we solve iteratively in Section \ref{subsec:numex}.

\section{Multigrid and preconditioning} \label{sec:prec}
%--------------------------------------
% Intro on CSL and CSG
%--------------------------------------
In this section we discuss the effect of the spectral properties of the indefinite Helmholtz operator on the convergence of iterative processes such as multigrid. The spectrum $\lambda_{H_h}\in\sigma(H_h)$ of the discrete Helmholtz operator and that of the discrete Laplacian on an ECS grid, i.e., $\lambda_{L_h}\in\sigma(L_h)$, varies only up to a real constant $-k^2$, which defines the distance by which an eigenvalue $\lambda_{L_h}$ shifts westwards to render $\lambda_{H_h}$. We saw in Section~\ref{sec:NAdisc} that $\lambda_{H_h}$ is constrained to an area consisting of a straight line segment, and a region bounded by a triangle, the so-called pitchfork. We observe that for small values of $hk$, \edt{the smallest eigenvalues} lie on the tail of the pitchfork. The eigenmodes corresponding to these small eigenvalues are the standing waves that cover \edt{both} the real and the complex part of the domain. It is important to note that none of the small eigenvalues can be exactly equal to zero because we have shown that each of the corresponding modes must possess at least a non-zero imaginary part. Depending on the magnitude of the real wave number, this can potentially lead to a very large (but bounded) condition number and thus confirms that the discrete problem is ill-conditioned.

\subsection{Multigrid}
In this section, we assume familiarity with basic geometric multigrid. See \cite{Briggs00, Wess92, Trot01} for a quick
access. Here, we briefly skim through the multigrid difficulties in solving an indefinite Helmholtz problem.

The first aspect of multigrid that requires attention in the context of indefinite linear systems, is the absence of a
pointwise smoothing procedure. For a given discrete operator $M_h$, a strict condition on the so-called
\textit{h-ellipticity} measure $E_h(M_h)$ \cite{Brandt80, Trot01}, viz, $E_h(M_h) > 0$ formally implies the existence
of a pointwise smoothing process. Circumventing the details, it suffices to mention here that the $h$-ellipticity of the
discretized indefinite Helmholtz operator is very close to zero for interesting values of the wave number, and
therefore, common stationary methods do not amply relax the error to be representable on the coarse grid.

The other troublesome multigrid aspect that merits attention is \edt{coarse grid} correction. To see this, imagine that we increase \edt{the mesh width while keeping $k$ constant. The rightmost eigenvalues of the discrete operator on the finest level lies near $4/h^2-k^2$ and cannot come near zero without strictly violating the posted accuracy condition $hk < 0.625$}. However, this is more likely to happen within a multigrid cycle where the same operator is re-discretized on the coarser grids, each with a larger mesh width. This effect can lead to resonant behavior \edt{on a level where $4/h^2-k^2 \approx 0$}. This leads to a severe degradation of multigrid performance. This issue is also well-known and discussed in papers, such as \cite{EEL01, EOV06}.

An important point to observe in this context, is that smoothing is much less of a trouble than coarse grid correction in the case of a multigrid solution of an indefinite problem. For moderate wave numbers in MP1 such as $k=40$, we see that multigrid still converges after a careful choice of components. For example, we conducted a test on this problem with $k=40$ and interior grid size $N=64$ \edt{in both dimensions}; with \edt{ECS layers} on all edges. The components were ILU(0) with $\omega=0.3$, F(1,2)-cycle, FW-restriction, bilinear prolongation and Galerkin formulation of the coarse grid operator. For this problem multigrid converged in $21$ iterations. However, we also saw that for twice the grid size and twice the wave number, the same algorithm failed to converge with any combination of components. \edt{When} we introduced a slight damping in the wave number, $k = (80-0.05\imath)$, multigrid converged again in $35$ F(1,1)-cycles, and thus confirms that it might be used for approximate preconditioner solves in a Krylov setup. \edt{Note that these choices of $k$ and $N$ are not ideal from an accuracy perspective, however, the above example is useful to understand the multigrid performance.}

\subsection{A short overview of the complex shifted Laplacian (CSL) preconditioner} \label{subsec:CSLprec}
The idea of preconditioning the Helmholtz problem with its (slightly) damped version as published by Erlangga \edt{et al.}\ in \cite{EVO04,EOV06,EVO06} is founded on avoiding the diverging behavior of multigrid for the original indefinite problem. The damping is brought about by a complex shift of the Laplacian, which evidently shifts its spectrum away from the origin. As a result, the discrete problems formulated with the shifted Laplacian can be tackled by multigrid. This preconditioning is perfectly applicable in the ECS context as well. The continuous version of the preconditioner for the Helmholtz problem on the ECS domain $[0,1]\cup[1,R_z]\subset \mathbb{C}$ is given by
\begin{equation}\label{eq:helmctucsl}
M^{CSL} \equiv -(L +\beta^2 k^2) \mbox{ \edt{in} } [0,1]\cup[1,R_z]\subset \mathbb{C},
\end{equation}
with a complex shift \edt{$\beta^2 = \varepsilon_1+\imath\varepsilon_2\in\mathbb{C}$. It is also important to know that in the comparison with \cite{EOV06}, the complex shift $\beta_1-\imath\beta_2$ in \cite{EOV06} is equivalent to $\varepsilon_1+\imath\varepsilon_2$ in this paper.
}

Since $M^{CSL} = -(L +\beta^2 k^2)$ and $H = -(L +k^2)$ share the same eigenvectors it is easy to see that the eigenvalues of the continuous preconditioned system $(M^{CSL})^{-1} H$ lie on a circle. Indeed, the spectrum is given by the linear fractional transformation $LF(\mu) = \frac{\mu+k^2}{\mu+\beta^2 k^2}$ that maps the complex line of eigenvalues $\lambda_{L}$ of the Laplacian to the circle through $1/\beta$, $\frac{1+R_z^2k^2}{1+R_z^2 \beta^2 k^2}$ and $1$.

The linear fractional transformation maps the negative imaginary half plane to the interior of the complex circle through the complex values $0$, $1$ and $1/\beta^2$. So the spectrum of the discretized operator $L_h$ is mapped to this region. More specifically, each line of the triangle that bounds the eigenvalues is mapped onto a segment of a circle that lies inside this region. As a result, the eigenvalues of the preconditioned discrete system $(M^{CSL}_h)^{-1}H_h$ lie away from the origin, inside a banana shaped figure that is the image of the different branch lines in the pitchfork and the tail from Figure~\ref{fig:pitchfork} as illustrated in Figure~\ref{fig:banana}.
% fig: banana
\begin{figure}[htp!]
\begin{center}
\includegraphics[width=7cm]{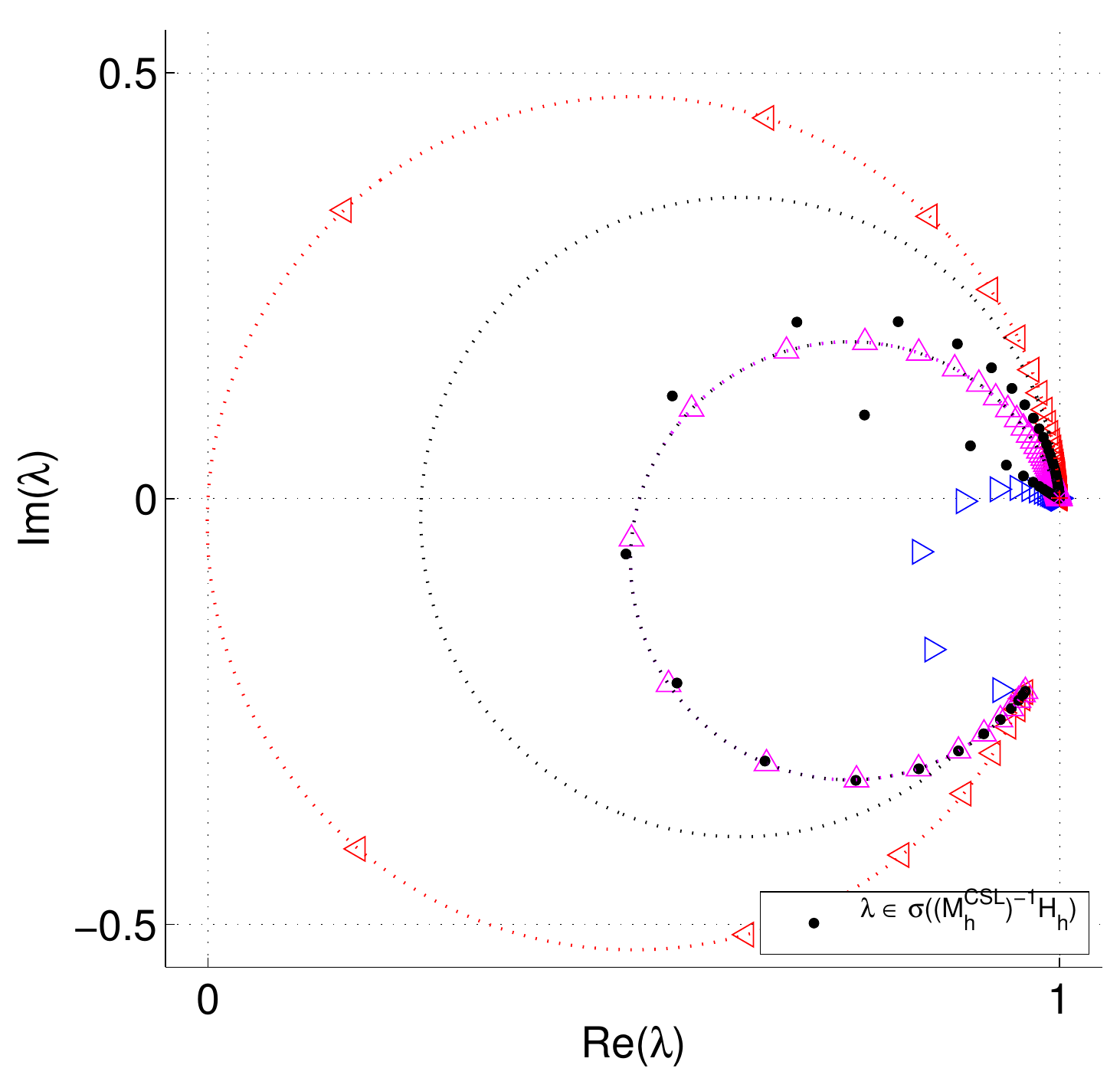}\caption{The pitchfork of the original Helmholtz operator in Figure~\ref{fig:pitchfork} is mapped to a banana (dotted lines), with the eigenvalues ($\bullet$) inside. The lines corresponding to the eigenvalue problem on the real domain and the complex contour are mapped to a circular arc ($\triangleleft$) through $0$ and a smaller circular arc ($\triangleright$) respectively. They enclose the preconditioned spectrum and the image of the intermediate complex line associated to the eigenvalue problem on the complex line $[0,R_z]$ ($\triangle$). (color online)}\label{fig:banana}
\end{center}
\end{figure}
The preconditioned system can be solved much more conveniently with Krylov subspace solvers by employing very few multigrid cycles for approximate preconditioner inversion.

\subsection{The complex \edt{stretched} grid (CSG) preconditioner} \label{subsec:CSGprec}
We readily see that the indefinite spectrum may be shifted favorably by an alternate strategy, which is more focused towards the ECS formulation. Instead of scaling the wave number, we keep it unchanged \edt{and scale} the discretization grid. To see this,
imagine discretizing the one-dimensional complex shifted Laplacian $M^{CSL} = -\frac{d^2}{d x^2}-\beta^2 k^2$
with $\beta\in\mathbb{C}$ on a real interval with constant mesh width $h\in\mathbb{R}$, and note that $(\beta^2 k^2)
h^2 = k^2(\beta^2 h^2)$. The left term in this last equality appears in the discretization of $M^{CSL}$,
while the right term can be interpreted as a quantity that appears in the discretization of the Helmholtz problem
$M_{CSG} = -\frac{d^2}{d z^2}-k^2$ defined on a straight complex line with constant mesh width \edt{$\beta h\in\mathbb{C}$}. Indeed
\edt{\begin{align}
M^{CSL}_h u_h \equiv -(\frac{1}{h^2}L_h+\beta^2 k^2)u_h = b_h &\Leftrightarrow M^{CSG}_h u_h \equiv -(\frac{1}{\beta^2 h^2}L_h+k^2)u_h = \frac{1}{\beta^2} b_h, \label{eq:CSL-CSG}
\end{align}}
so the system \edt{$M^{CSL}_h u_h = b_h$} yields the same solution as \edt{$M^{CSG}_h u_h = b_h/\beta^2$}. \edt{In this way we found the equivalent \emph{complex stretched grid (CSG)} preconditioner $M^{CSG}$. In this context we will denote the angle of $\beta$ in the complex plane as $\theta_\beta$.}

The same argument still holds on ECS domains where the contour mesh widths are \edt{already complex} and for inhomogeneous wave numbers as present in MP2 and MP3. This approach offers \edt{extra} possibilities to explore. Instead of scaling, for example, the entire spectrum away from the origin, only the problematic branch of eigenvalues close to the real axis can be scaled deeper into the complex plane. In other words, only the interior part of the grid is scaled with $\beta$, while the complex contour stays the same. In general, we can build a preconditioner by defining the original Helmholtz equation on a convenient domain, given by a coordinate transformation
\edt{\begin{equation*}
 z(x) = \left\{
  \begin{array}{ll}
    x+\imath f_\Omega(x), & \hbox{$x \in \Omega$;} \\
    x+\imath f_\Gamma(x), & \hbox{$x \in \Gamma$,}
  \end{array}
\right.
\end{equation*}}
with $f_\Omega,f_\Gamma \in \mathcal{C}^2$ increasing (e.g. linear, quadratic, \ldots) and $\displaystyle \lim_{x\to\partial\Omega}f_\Omega(x) = \displaystyle \lim_{x\to\partial\Omega}f_\Gamma(x)\neq0$.

Applied to \eqref{eq:lapctuecs} we get
\begin{equation}\label{eq:lapctucsg}
M^{CSG}u(z) \equiv -\frac{\partial^2}{\partial z^2} u(z) \mbox{ \edt{in} }
[0,z(1)]\cup[z(1),z(R)]\subset \mathbb{C},
\end{equation}
with homogeneous Dirichlet boundary conditions in $z(0)=0$ and $z(R)\equiv R_z$. Independent of the complex contour $\Gamma_{z} = [z(1),R_z]$, we now have the interior region $\Omega_z=[0,z(1)]$ complex as well.\\
\edt{In Figure~\ref{fig:ctucsg} the interior region is chosen along the line connecting the two boundaries $0$ and $R_z$ of the original ECS domain. As a result the interior mesh width is scaled from $h$ to $\beta h$; the complex contour has no extra scaling, the mesh width $\gamma h$ is preserved. Figure~\ref{fig:frenchdonut} displays the resulting preconditioned system \edt{with CSG (right) versus the CSL (left) with the same scaling in the wave number, from $k$ to $\beta k$. CSG leads to a similar} C-shaped spectrum, away from the origin, favorable for Krylov methods. However, experiments show that the spectrum of the preconditioner (see Figure~\ref{fig:pitchforkcsg}) is still bad for our current multigrid configuration. The numerical experiments with the CSG preconditioning method in the next section all use grids that are equally scaled over the entire domain, i.e.\ the interior part and complex contour. The CSG preconditioner with scaling factor $\beta$ is then related to the equivalent CSL preconditioner with scaled wave number $\beta k$, this means a complex shift $\beta^2$, as in equation \eqref{eq:CSL-CSG}.}
% fig: fig_ctucsg
\begin{figure}[htp!]
\begin{center}
\includegraphics[width=9cm]{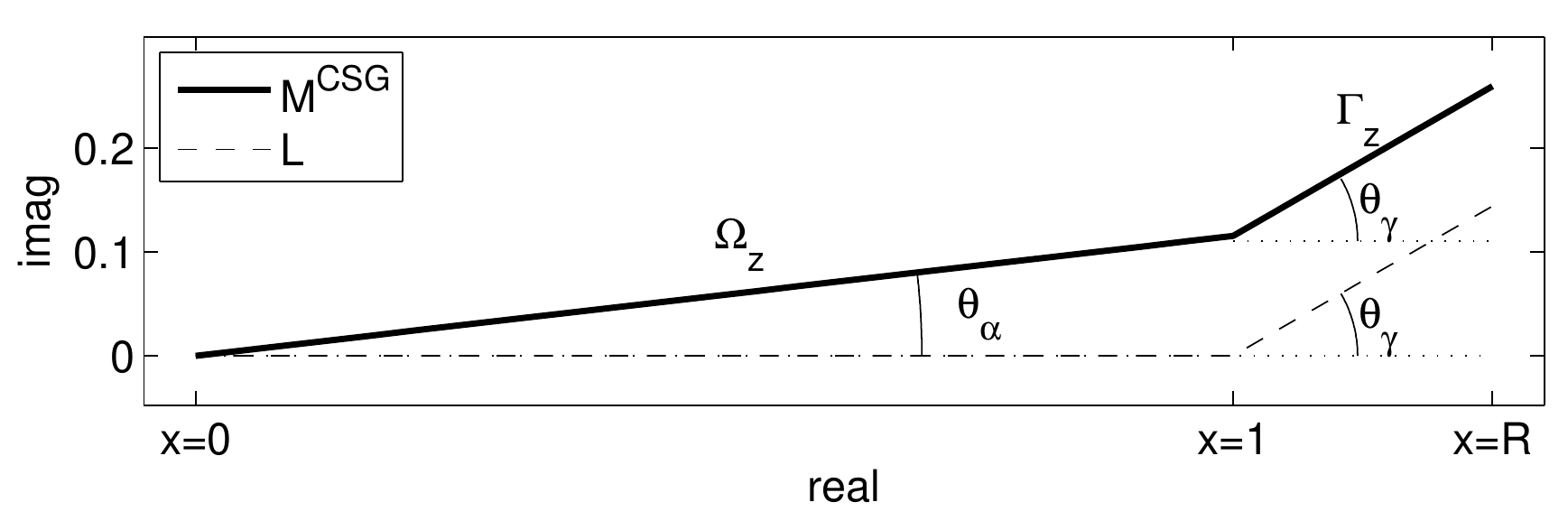}\caption{The CSG domain. The domain for $M^{CSG}$ (line) is complex \edt{stretched} in the region of interest $[0,1]$ as well, \edt{such that it aligns with} the line connecting the two boundaries $0$ and $R_z$ of the original ECS domain (dashed line).}\label{fig:ctucsg}
\end{center}
\end{figure}
% fig: pitchforkcsg
\begin{figure}[htp!]
\begin{center}
\includegraphics[width=9cm]{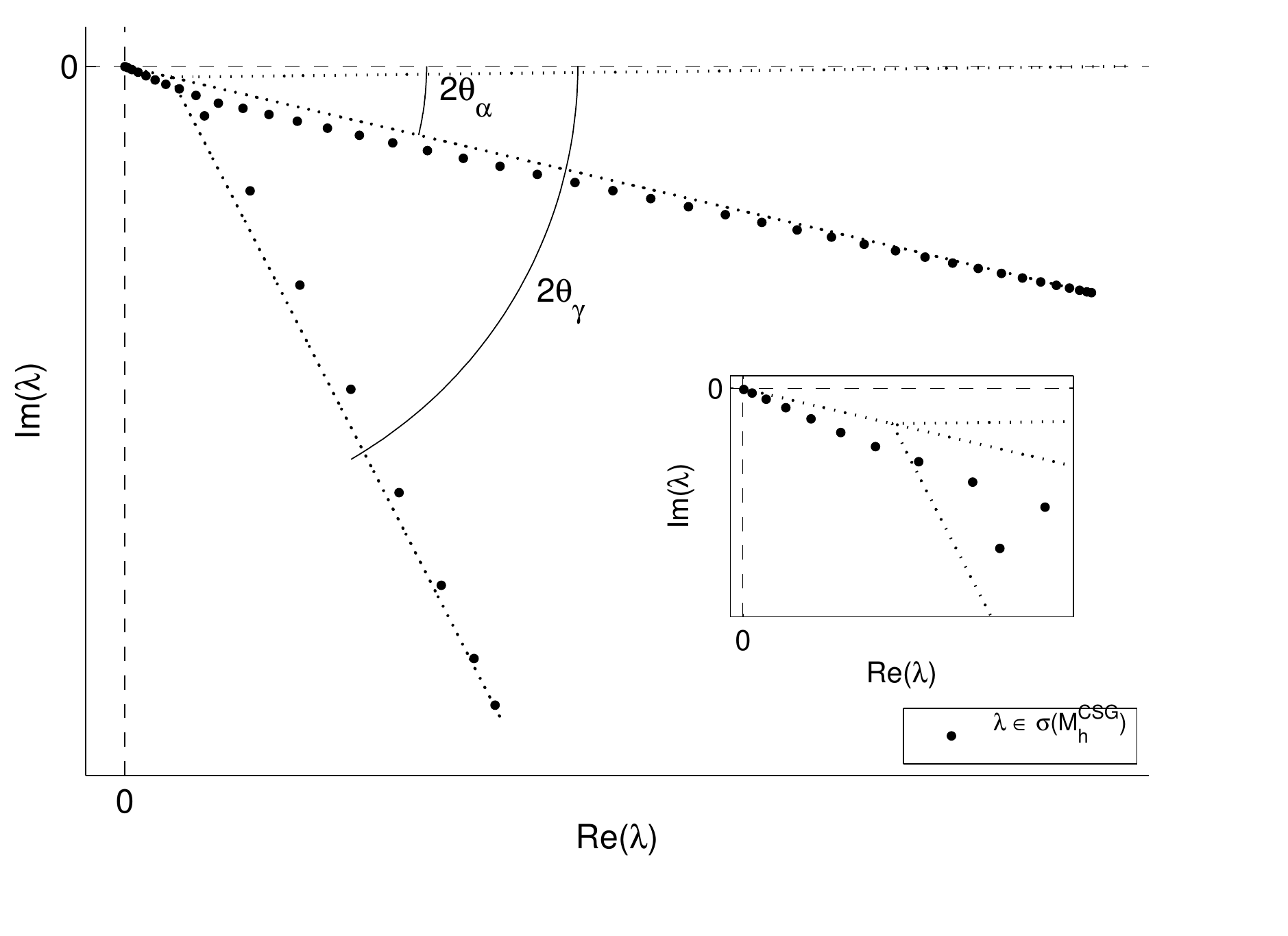}\caption{The eigenvalues of the discrete preconditioner $M_h^{CSG}$ on the domain in Figure~\ref{fig:ctucsg}. The eigenvalues lie in a narrow pitchfork ($\bullet$) that lies in the bottom half of the pitchfork of the original matrix $-L_h$ (dotted lines) from Figure~\ref{fig:pitchfork}. By choosing the domain as in Figure~\ref{fig:ctucsg}, the top branch of the new pitchfork lies along the middle line associated to the eigenvalue problem on the complex line $[0,R_z]$, i.e. the line of eigenvalues of the continuous problem.}\label{fig:pitchforkcsg}
\end{center}
\end{figure}
% fig: frenchdonut
\begin{figure}[htp!]
\begin{center}
\subfigure[Complex shifted Laplacian]{\includegraphics[width=6cm]{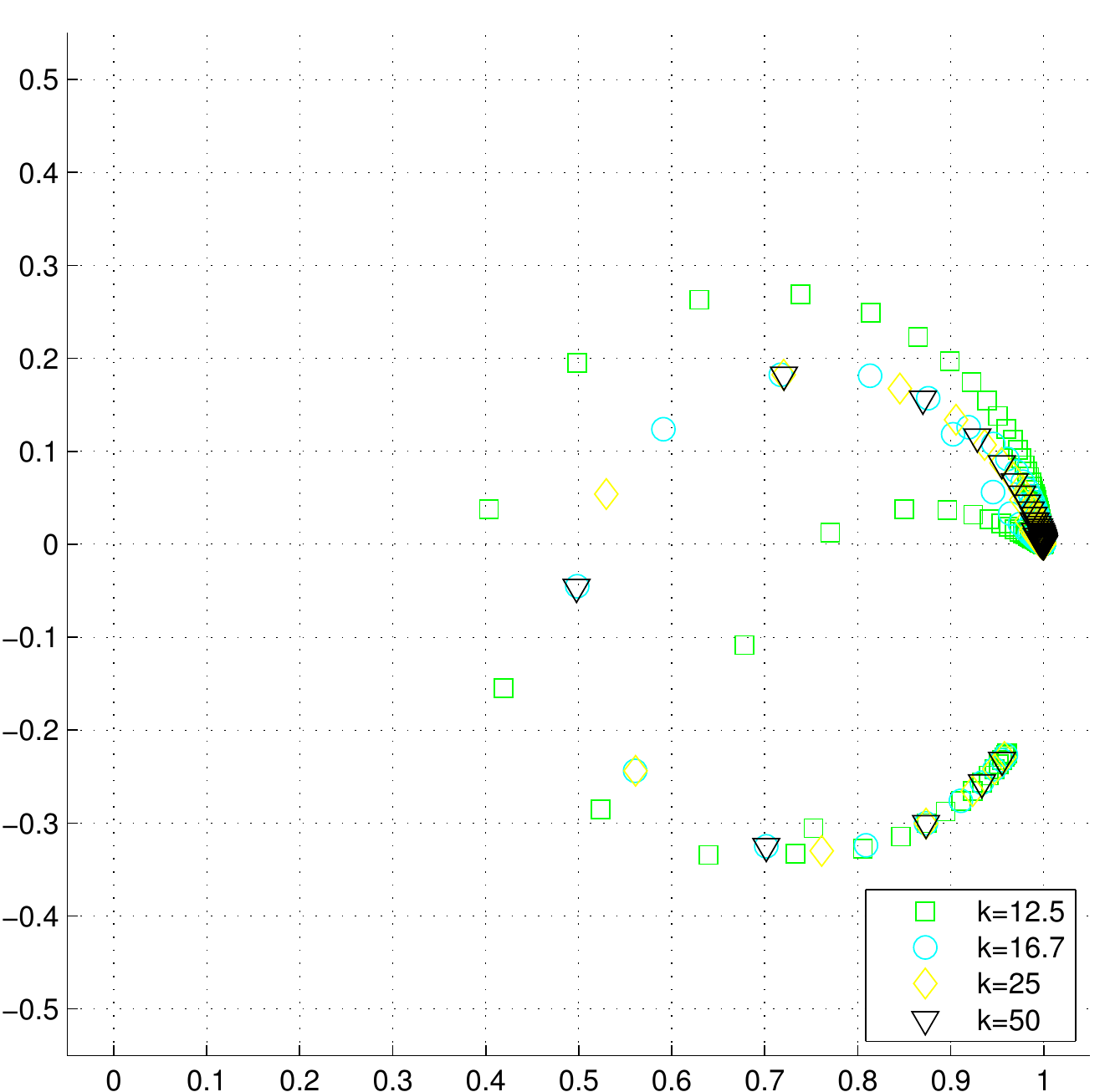}}
\subfigure[Complex stretched grid]{\includegraphics[width=6cm]{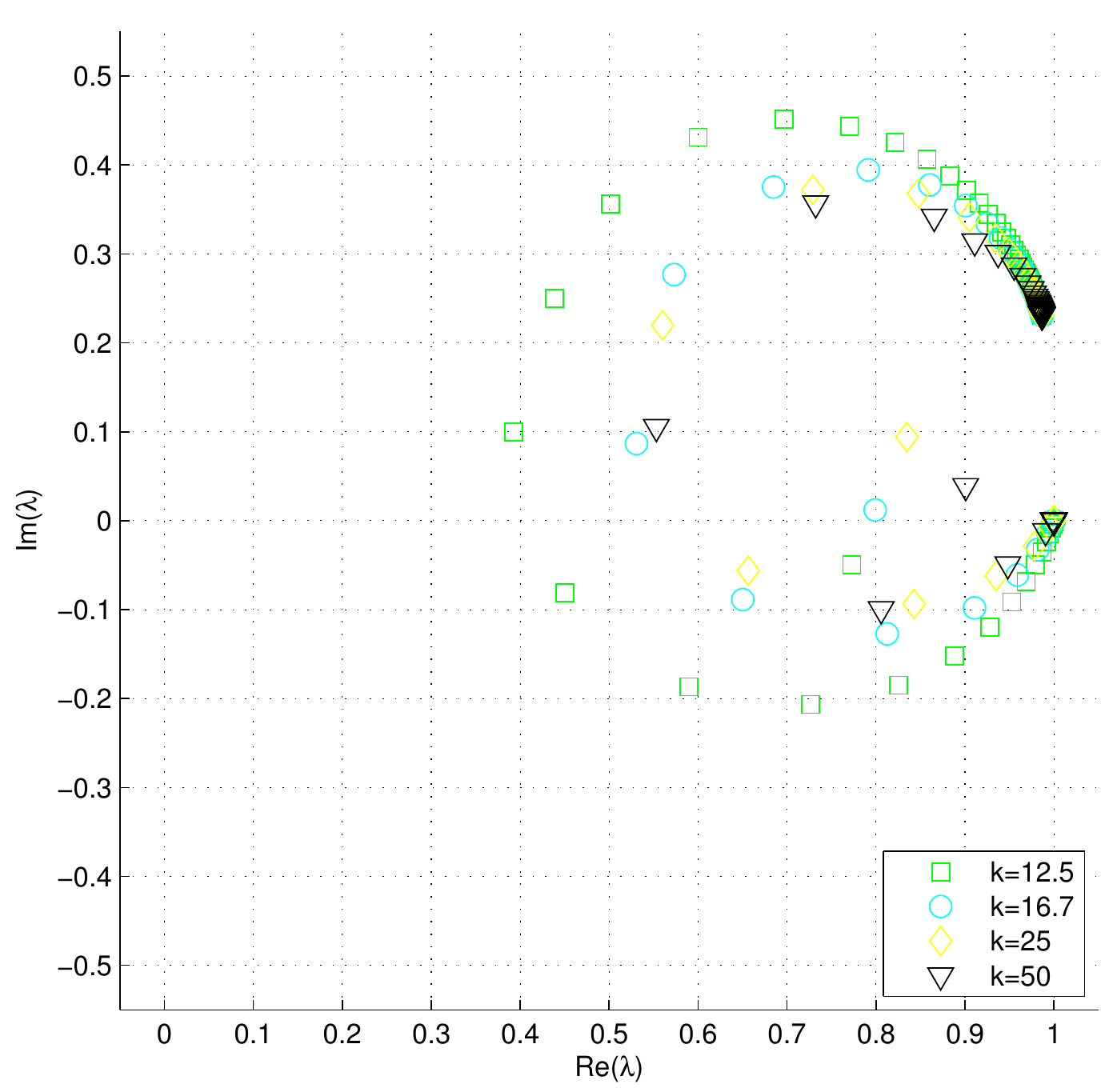}}
\caption{\edt{Left: The banana shaped spectrum of the CSL-preconditioned Helmholtz system $(M_h^{CSL})^{-1}H_h$ in Figure~\ref{fig:banana} grows slightly towards the origin for increasing $k$. The same effect was observed for Sommerfeld radiation conditions by Erlangga et al.\ in \cite{EVO06}. Right: The eigenvalues of the equivalent discrete preconditioned Helmholtz system $(M_h^{CSG})^{-1}H_h$ with the complex stretched grid preconditioner defined on the domain in Figure~\ref{fig:ctucsg}, lie on a C-shaped figure, away from the origin, with a similar dependence on $k$.}}\label{fig:frenchdonut}
\end{center}
\end{figure}

\section{Numerical experiments}\label{sec:experiments}
%--------------------------------------
% experiments: intro
%--------------------------------------
In this section we supplement the theoretical development with numerical experiments carried out on the model problems with both kinds of boundary treatment as discussed earlier. We provide a brief comparison of different multigrid components, which allowed us to choose the best set for the approximate multigrid inversion of the preconditioning operator during a Krylov solve. The details of the solver and results from the numerical experiments are displayed in tables and figures for easy access. For the numerical experiments we use the Shortley-Weller finite difference discretization \edt{\eqref{eq:shortwell}} applied to a cell-centered mesh topology. The cell-centered mesh topology is chosen because multigrid is slightly more convenient with this choice for general Robin-type first order derivative boundary conditions (henceforth BC), of which the Sommerfeld radiation BC are a special case. Moreover, the stencil is the same as for the vertex-centered case, so the results are easily carried over.

\subsection{Choice of Multigrid Components} \label{subsec:mg}
%--------------------------------------
% experiments: multigrid
%--------------------------------------
To select a set of multigrid components from the different available choices, we take MP1 with $k=80$, and use a grid size of $128^2$ obeying the accuracy constraints. ECS \edt{layers} are used, and defined by scaling the mesh \edt{widths} in these layers by \edt{$e^{\imath \theta_\gamma}$} with the angle $\theta_\gamma = \pi/6$. In this case, most of the spectrum lies in the $4^{th}$ quadrant of the complex plane, i.e., except the eigenvalues responsible for making the linear system indefinite. Through a negative imaginary shift of the Helmholtz operater equal to $-0.2$, we push the spectrum adequately towards the $4^{th}$ quadrant, thus transforming the linear system so that it is now nearly \emph{negative semi-definite}. It is imperative to comprehend that with Sommerfeld BC, the major part of the
spectrum of the original problem is in the $1^{st}$ quadrant, and therefore with Sommerfeld BC, a preconditioner formed
through a positive imaginary shift will make sense.

% fig: cyclecomp and timecomp
\begin{figure}
\subfigure[Defect reduction / mg cycles]{\includegraphics[width=7.5cm]{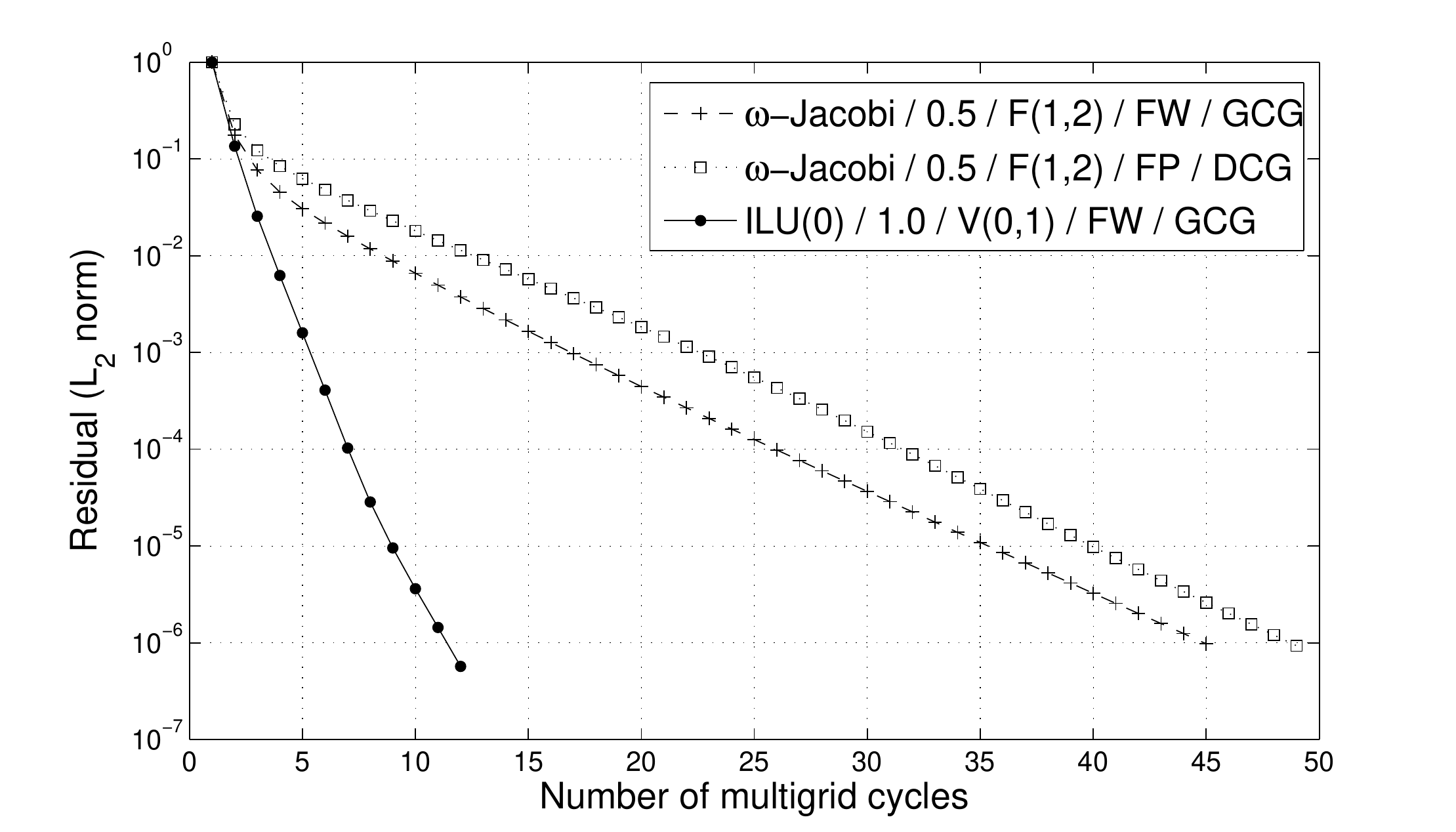}}
\subfigure[Defect reduction / CPU time]{\includegraphics[width=7.5cm, height=4.2cm]{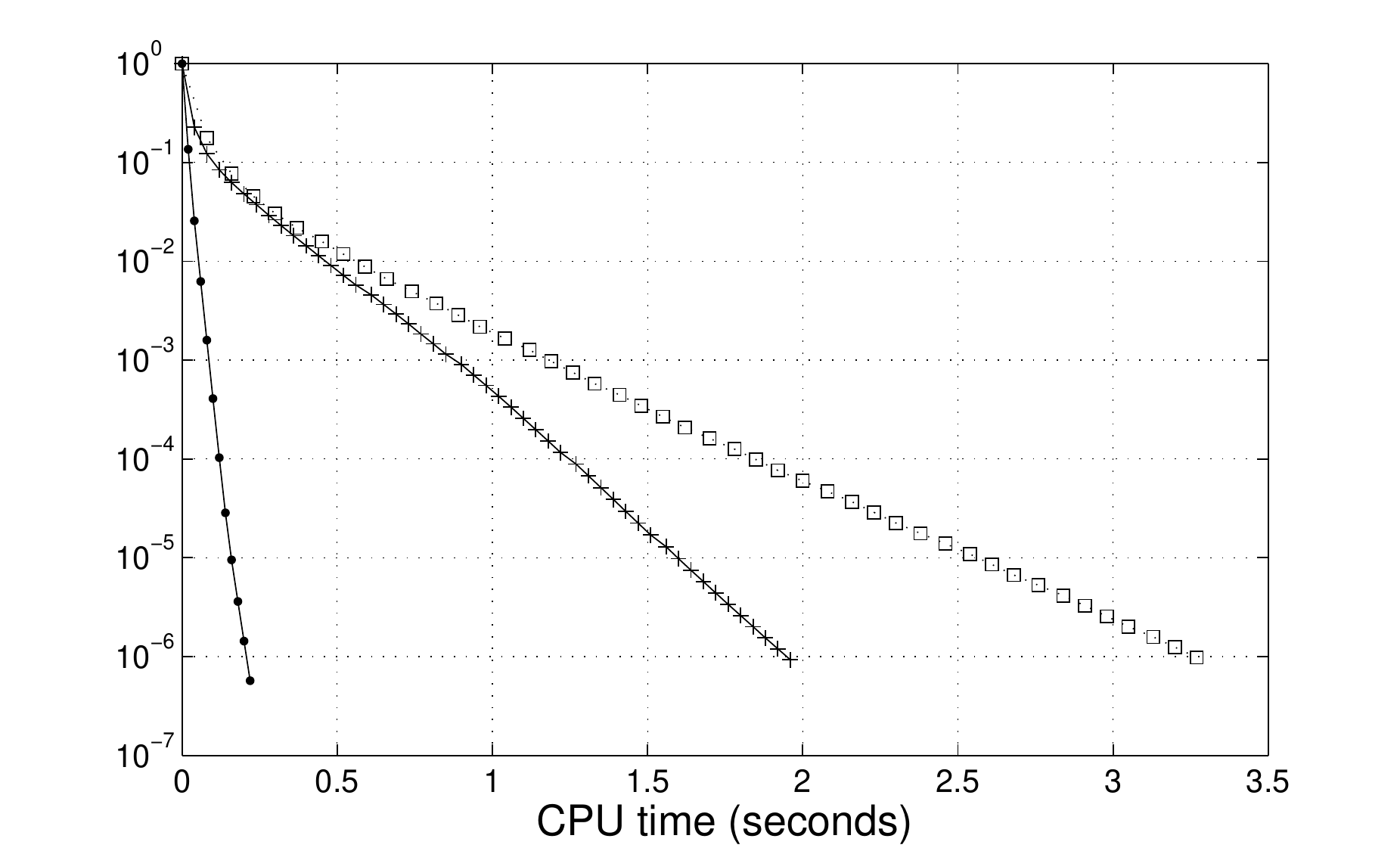}}
\caption{Multigrid performance on the preconditioner formed through the CSL technique ($\beta = (1,-0.2)$) for MP1, with $k = 80$, $128$ interior points plus $2 \times 32$ points in the ECS \edt{layers} along each dimension (on either sides). In all the cases, prolongation was done through bilinear interpolation. The legend incorporates the smoother, the relaxation parameter, the restriction method, the
cycle type, and the coarse grid operator. DCG stands for Direct Coarse Grid operator, while GCG stands for the Galerkin
Coarse Grid operator.} \label{fig:mgMP1prec}
\end{figure}

From Figure \ref{fig:mgMP1prec}, we immediately recognize the set of multigrid components that work best in the present
context. ILU(0) smoothing, Full Weighted (FW) averaging as restriction, bilinear interpolation as prolongation, and the
Galerkin formulation for the coarse grid operator. These components are harnessed within a V(0,1)-cycle and yield an
algorithm that does an excellent job on the preconditioner. Therefore, these are the components that we will invariably
use in Section \ref{subsec:numex}. For the sake of completeness, we also checked out multigrid performance with various
combinations of $\omega$-Jacobi, with an under-relaxation of $0.5$, employed in F-cycles. The comparison was
very thorough, and included tests with the Four-Point averaging \cite{Trot01} as restriction, as well as with direct
discretization on the coarse grids. The results are presented in Figure \ref{fig:mgMP1prec}, are self explanatory, and
indicate the most viable set of components very clearly.

\subsection{Numerical Experiments} \label{subsec:numex}
%--------------------------------------
% experiments: results
%--------------------------------------
In this section before putting up the results of the numerical experiments, we first sort out the minimum grid size
requirement for MP2 and MP3, which are Helmholtz problems with strongly varying wave numbers. As the wave numbers are
infinite at the origin, we take into consideration the highest discrete wave number that the discretized problem can
attain under the cell-centered mesh topology. For reasons of brevity, the complete analysis is not shown here. We
however, brief the steps involved in the analysis. First, we transform the problem into a new coordinate system; so
that $\widetilde{x} = x/r$, and $\widetilde{y} = y/r$. This new system is in the unit square
domain and is therefore dimensionless. The next step is to evaluate the supremum of the transformed wave number, which
\edt{occurs at the origin for the model problems}. For MP2, this is \edt{straightforward}. We use this value in the accuracy condition and are
led to the maximum mesh \edt{widths} that must be obeyed. For MP2 this is given in Table \ref{tab:meshsizesMP2}.

\begin{table}
\centering
\begin{tabular}{|c c|c|c|c|c|c|}
\hline
 & $k$ & 1 & 2 & 3 & 4 & 5 \\
$\nu$ & & & & & & \\ \hline 0  & & 0.625 & 0.312 & 0.208 & 0.156 & 0.125 \\ \hline 1  & & 0.360 & 0.255 & 0.188 &
0.147 & 0.120 \\ \hline 2  & & 0.279 & 0.221 & 0.173 & 0.140 & 0.116 \\ \hline 3  & & 0.236 & 0.197 & 0.161 & 0.133 &
0.112 \\ \hline 4  & & 0.208 & 0.180 & 0.151 & 0.127 & 0.109 \\ \hline 5  & & 0.188 & 0.167 & 0.143 & 0.122 & 0.105 \\
\hline 6  & & 0.173 & 0.156 & 0.136 & 0.118 & 0.102 \\ \hline 7  & & 0.161 & 0.147 & 0.130 & 0.114 & 0.100 \\ \hline 8
& & 0.151 & 0.139 & 0.125 & 0.110 & 0.097 \\ \hline 9  & & 0.143 & 0.133 & 0.120 & 0.107 & 0.095 \\ \hline 10 & & 0.136
& 0.127 & 0.116 & 0.104 & 0.093 \\ \hline
\end{tabular}
 \caption{Maximum mesh \edt{width} limits for \edt{MP2}}
 \label{tab:meshsizesMP2}
\end{table}
Here, we will solve two test-cases of MP2. One is characterized by $\nu = 7$ and $k=2$, while the other is
characterized by $\nu = 1$ and $k=4$. For both of these test-cases, the minimum grid size (from a multigrid
perspective) is $384^2$.

For MP3, we immediately see that the origin cannot be substituted into the dimensionless wave number (due to the
singularity). We observe that ($h_x/2, h_y/2)$ gives the largest wave number that the discrete
problem can attain. Using this in the accuracy condition results in a quadratic equation, which we solve to get the mesh
sizes. These are scaled back to $(0,r)^2$ and are given in Table \ref{tab:meshsizesMP3}.

\begin{table}
\centering
\begin{tabular}{|c c|c|c|c|c|c|}
\hline
 & $k$ & 1 & 2 & 3 & 4 & 5 \\
$r$ & & & & & & \\ \hline \edt{$\forall ~50<r<200$}  & & 0.095 & 0.089 & 0.082 & 0.075 & 0.068 \\ \hline
\end{tabular}
 \caption{Upper limit on the mesh \edt{width} for MP3 with respect to \edt{$k$}}
 \label{tab:meshsizesMP3}
\end{table}

We form two test-cases with MP3 as well. The first one is characterized by $r=90$ and $k=2$ and requires a minimum
interior grid size of $1024^2$. The second test case is much more severe. It is formed by $r=150$ and $k=4$, and
requires a minimum interior grid size of $2048^2$ for an acceptable resolution.

\begin{table}
%\hspace{-5mm}
\small
\edt{
\begin{tabular}{c|c|l||c|c|c|c|c|} \cline{2-7}
& \multirow{4}{*}{BC} & \multirow{4}{*}{Preconditioning operator}                    & MG perform.,  &   MGP &      MGP    &  MGP \\
&                     &                                                              & on precond.   & Bi-CGSTAB & IDR(4)  & IDR(8) \\ \cline{4-7}
&                     &                                                              & Conv., \# iter &   matvec  & matvec  & matvec \\
&                     &                                                              &   cputime     &  cputime  & cputime & cputime \\ \hline \hline
\multirow{12}{*}{\begin{sideways}MP1: $k=160$\end{sideways}}
& \multirow{6}{*}{\begin{sideways} Sommerfeld BC \end{sideways}}
  & \multirow{2}{*}{CSL, with shift = $-1 + 0.2\imath$}                              & 0.31,~~12 &   79      &    72     &  76        \\
& &                                                                                  & 0.52 sec. & 4.75 sec. & 4 sec.    & 5.35 sec.  \\ \cline{3-7}
& & \multirow{2}{*}{CSG, with angle = $\displaystyle-\frac{\pi}{28}$}                & 0.24,~~10 &   74      &  70       & 72         \\
& &                                                                                  & 0.47 sec. & 4 sec.    & 3.98 sec. & 4.40 sec.  \\ \cline{3-7}
& & \multirow{2}{*}{CSL, with shift = $-e^{-2\imath\nicefrac{\pi}{28}}=-0.97+0.22\imath$} & 0.26,~~11 &   74      &  69       & 72         \\
& &                                                                                  & 0.48 sec. & 4.78 sec. & 4.26 sec. & 5.20 sec.  \\ \cline{2-7}\\[-2.8ex]\cline{2-7}
& \multirow{6}{*}{\begin{sideways} ECS Layers \end{sideways}}
  & \multirow{2}{*}{CSL, with shift = $-1-0.2\imath$}                               & 0.29,~~12 &    59     &    62     & 60         \\
& &                                                                                  & 1.13 sec. & 9.15 sec. & 9.11 sec. & 9.61 sec.  \\ \cline{3-7}
& & \multirow{2}{*}{CSG, with angle = $\displaystyle \frac{\pi}{28}$}                & 0.24,~~10 &    58     &    62     & 56         \\
& &                                                                                  & 1.12 sec. & 9 sec.    & 9.23 sec. & 9.62 sec.  \\ \cline{3-7}
& & \multirow{2}{*}{CSL, with shift = $-e^{2\imath\nicefrac{\pi}{28}}=-0.97-0.22\imath$}  & 0.24,~~10 &    58     &    59     & 54         \\
& &                                                                                  & 1.10 sec. & 9.04 sec. & 9.06 sec. & 9.30 sec.  \\ \cline{2-7}
 \end{tabular}
\caption{Experimental results - Iterative solution of MP1. One $V(0,1)$ multigrid cycle is used for approximate inversion of the preconditioners. This preconditioning is used with Bi-CGSTAB, IDR(4) and IDR(8). The interior domain is a unit square, discretized with $256$ points along both dimensions. With ECS layers, there are $2 \times 64$ additional points per dimension, belonging to the layers. With ECS formulation each edge of the domain is endowed with an absorbing layer.}}
\label{tab:MP1}
\end{table}

\edt{\begin{rem}[Reading and comparing the experimental results] The experimental results are summarized in Tables \ref{tab:MP1}, \ref{tab:MP2}, and \ref{tab:MP3}. Table \ref{tab:MP1} accounts for one, and Tables \ref{tab:MP2} and \ref{tab:MP3} account for two problems (each) derived from using different values of the parameters in the model problems. This is mentioned vertically in the first column of each table. Each of these problems give two different discrete versions, when formulated once with the Sommerfeld approximation (on the boundaries), and second, when formulated with ECS layers. This is also vertically marked in the tables. Next, each discrete formulation is solved with Krylov methods, using three preconditioners; (1) the CSL preconditioner with the shift $k^2\rightarrow\beta^2k^2$ where $\beta^2 = \varepsilon_1 + \imath\varepsilon_2\equiv1+\varepsilon_2$ (with $\varepsilon_2$ having the smallest absolute value for which the chosen multigrid method converges), (2) the CSG preconditioner having the scaling $h\rightarrow\beta h$ where $\beta = e^{\imath\theta_\beta}$ (with $\theta_\beta$ the smallest angle for which the chosen multigrid method converges), and finally, (3) the CSL preconditioner again, with a complex shift equal to $\beta^2 = e^{2\imath\theta_\beta}$ (with $\theta_\beta$ as used with the CSG preconditioner). Each row in the tables accounts for one of the above mentioned preconditioners. Each row (after specifying the preconditioner) lists the following information:
\begin{enumerate}
	\item Multigrid performance on the preconditioner taken as a standalone problem. \emph{Conv.}, is the average convergence factor per cycle, and \emph{\#iter} is the number of cycles consumed to reduce the relative residual by 7 orders of magnitude.
	\item The number of multigrid preconditioned (MGP) matvecs employed by Bi-CGSTAB. Note that each iteration of Bi-CGSTAB consists of two such matvecs. We chose matvecs over conventional iterations to have a better idea of the total number of multigrid cycles used, as well as to have a fair comparison with IDR($s$).
	\item The number of multigrid preconditioned matvecs employed by IDR(4).
	\item The number of multigrid preconditioned matvecs employed by IDR(8).
\end{enumerate}
Within each row, the three Krylov methods can be compared against one another with the same preconditioner. Within the stacks containing three rows each, the relative performance of the three preconditioners on an identical discrete problem can be checked.
\end{rem}}

In this paper we use one multigrid $V(0,1)$-cycle for each preconditioning step that is involved
in the solver. Convergence of the Krylov solver is determined by a check on the relative residual going below a tolerance
value of $10^{-6}$, i.e., the algorithm stops after the $m^{th}$ iteration if:
\begin{align}
\frac{ \lVert d_m \rVert_2}{\lVert d_0 \rVert_2} < 10^{-6}
\end{align}
$\lVert d_i \rVert_2$ is the defect (measured in the discrete $L_2$-norm) after the $i^{th}$ solver iteration.

\begin{table} [!h]
%\hspace{-5mm}
\small
\edt{
\begin{tabular}{c|c|l||c|c|c|c|c|} \cline{2-7}
& \multirow{4}{*}{BC} & \multirow{4}{*}{Preconditioning operator}                    & MG perform.,  &   MGP &      MGP    &  MGP \\
&                     &                                                              & on precond.   & Bi-CGSTAB & IDR(4)  & IDR(8) \\ \cline{4-7}
&                     &                                                              & Conv., \# iter &   matvec  & matvec  & matvec \\
&                     &                                                              &   cputime     &  cputime  & cputime & cputime \\ \hline \hline
\multirow{12}{*}{\begin{sideways}MP2: $\nu=7, k=2$\end{sideways}}
& \multirow{6}{*}{\begin{sideways} Sommerfeld BC \end{sideways}}
  & \multirow{2}{*}{CSL, with shift = $-1 + 0.35\imath$}                             & 0.25,~~10 &   123      &    120     &  122      \\
& &                                                                                  & 1.10 sec. & 18.8 sec.  &  17.5 sec. & 20.0 sec. \\ \cline{3-7}
& & \multirow{2}{*}{CSG, with angle = $\displaystyle-\frac{\pi}{20}$}                & 0.25,~~10 &   113      &  115       & 115       \\
& &                                                                                  & 1.08 sec. & 15.8 sec.  & 14.6 sec.  & 16.2 sec. \\ \cline{3-7}
& & \multirow{2}{*}{CSL, with shift = $-e^{-2\imath\nicefrac{\pi}{20}}=-0.95+0.31\imath$} & 0.26,~~11 &   115      &  114       & 113       \\
& &                                                                                  & 1.20 sec. & 16.8 sec.  & 16.0 sec.  & 17.1 sec. \\ \cline{2-7}\\[-2.8ex]\cline{2-7}
& \multirow{6}{*}{\begin{sideways} ECS Layers \end{sideways}}
  & \multirow{2}{*}{CSL, with shift = $-1-0.34\imath$}                               & 0.17,~~8  &    71     &    72     & 74         \\
& &                                                                                  & 2.70 sec. & 31.2 sec. & 31.3 sec. & 32.3 sec.  \\ \cline{3-7}
& & \multirow{2}{*}{CSG, with angle = $\displaystyle \frac{\pi}{20}$}                & 0.19,~~9  &    68     &    73     & 69         \\
& &                                                                                  & 2.60 sec. & 31.5 sec. & 32.1 sec. & 33.2 sec.  \\ \cline{3-7}
& & \multirow{2}{*}{CSL, with shift = $-e^{2\imath\nicefrac{\pi}{20}}=-0.95-0.31\imath$}  & 0.18,~~9  &    70     &    77     & 71         \\
& &                                                                                  & 2.61 sec. & 32.9 sec. & 34.0 sec. & 33.8 sec.  \\ \hline \hline \hline
\multirow{12}{*}{\begin{sideways}MP2: $\nu=1, k=4$\end{sideways}}
& \multirow{6}{*}{\begin{sideways} Sommerfeld BC \end{sideways}}
  & \multirow{2}{*}{CSL, with shift = $-1+0.27\imath$}                               & 0.39,~~15 &   179     &   175     &  177       \\
& &                                                                                  & 1.60 sec. & 29.3 sec. & 27.2 sec. & 29.5 sec.  \\ \cline{3-7}
& & \multirow{2}{*}{CSG, with angle = $\displaystyle-\frac{\pi}{27}$}                & 0.51,~~21 &   161     &  164      & 161        \\
& &                                                                                  & 2.30 sec. & 26.2 sec. & 27.0 sec. & 27.1 sec.  \\ \cline{3-7}
& & \multirow{2}{*}{CSL, with shift = $-e^{-2\imath\nicefrac{\pi}{27}}=-0.97+0.23\imath$} & 0.53,~~22 &   161     &  160      & 159        \\
& &                                                                                  & 2.39 sec. & 26.1 sec. & 24.7 sec. & 26.8 sec.  \\ \cline{2-7}\\[-2.8ex]\cline{2-7}
& \multirow{6}{*}{\begin{sideways} ECS Layers \end{sideways}}
  & \multirow{2}{*}{CSL, with shift = $-1-0.27\imath$}                               & 0.34,~~13 &    138    &    125    & 131        \\
& &                                                                                  & 4.20 sec. & 43.4 sec. & 40.0 sec. & 44.0 sec.  \\ \cline{3-7}
& & \multirow{2}{*}{CSG, with angle = $\displaystyle \frac{\pi}{26}$}                & 0.43,~~18 &    115    &    123    & 114        \\
& &                                                                                  & 5.50 sec. & 36.0 sec. & 38.3 sec. & 37.9 sec.  \\ \cline{3-7}
& & \multirow{2}{*}{CSL, with shift = $-e^{2\imath\nicefrac{\pi}{26}}=-0.97-0.24\imath$}  & 0.44,~~17 &    116    &    120    & 124        \\
& &                                                                                  & 5.50 sec. & 35.8 sec. & 38.1 sec. & 39.3 sec.  \\ \cline{2-7}
 \end{tabular}
\caption{Experimental results - Iterative solution of two different discrete problems obtained from MP2 (by using different values of $\nu$ and $k$). One $V(0,1)$ multigrid cycle is used for approximate inversion of the preconditioners. This preconditioning is used with Bi-CGSTAB, IDR(4) and IDR(8). For both problems, the interior domain is a square of 50 units, discretized with $384$ points along both dimensions. With ECS layers, there are $128$ additional points per dimension, belonging to the layers. Contrary to MP1 the formulation with ECS Layers only has these layers on two edges of the domain, the north and the east.}}
\label{tab:MP2}
\end{table}

\edt{It is a well-known fact that due to indefiniteness multigrid cannot work directly with the Helmholtz problem as a solver. The CSL and the CSG preconditioners are attempts to maneuver the spectrum slightly so that its indefiniteness can be reduced, and that the eigenvalues may be slightly shifted so as to bring them closer to the positive or to the negative definite regimes in the complex plane. Roughly speaking, the CSL \emph{translates} the spectrum, while the CSG \emph{rotates} it to accomplish the objective. A preconditioner, therefore, can also be build as a hybrid between the CSL and the CSG approaches, i.e., with a general CSL shift $\varepsilon_1 + \imath\varepsilon_2$ combined with a general CSG rotation angle $\theta_\beta$. However, during experimentation, the latter preconditioner did not prove any better than either of these two approaches used in isolation, and is hence not presented.}

\edt{\begin{rem}[On choosing the CSL and the CSG parameters]
An automation can be set up which starts from a given imaginary shift for the CSL preconditioner, or a given angle for the CSG precondtioner, and monitors the residual norm obtained after successive multigrid cycles. In such an automatic routine, absolute values of the shift size or the angles may be reduced to a benchmark for which the given multigrid method just converges (say in 10-20 cycles). Parametrized with this shift (or angle), the preconditioner may be  used in the Krylov method with approximate solves for preconditioning. Note that such a selection routine may only run once, and decide upon the shifts to be employed for all later Krylov iterations. In this paper, however, we just resort to doing the above described process manually.
\end{rem}}

We first observe from Table \ref{tab:MP1} that the performance of the CSL preconditioner with both shifting strategies is similar \edt{to the performance of the CSG preconditioner}. The multigrid inversion of the preconditioners as well as the overall numerical solution method are very good for MP1.

\begin{table} [!h]
%\hspace{-5mm}
\small
\edt{
\begin{tabular}{c|c|l||c|c|c|c|c|} \cline{2-7}
& \multirow{4}{*}{BC} & \multirow{4}{*}{Preconditioning operator}                    & MG perform.,  &   MGP &      MGP    &  MGP \\
&                     &                                                              & on precond.   & Bi-CGSTAB & IDR(4)  & IDR(8) \\ \cline{4-7}
&                     &                                                              & Conv., \# iter &   matvec  & matvec  & matvec \\
&                     &                                                              &   cputime     &  cputime  & cputime & cputime \\ \hline \hline
\multirow{12}{*}{\begin{sideways}MP3: $r=90, k=2$\end{sideways}}
& \multirow{6}{*}{\begin{sideways} Sommerfeld BC \end{sideways}}
  & \multirow{2}{*}{CSL, with shift = $-1 + 0.38\imath$}                             & 0.44,~~18 &   207    &    214    &  232      \\
& &                                                                                  & 15.3 sec. & 4m 32s   &  4m 25s   & 4m 58s    \\ \cline{3-7}
& & \multirow{2}{*}{CSG, with angle = $\displaystyle-\frac{\pi}{17}$}                & 0.41,~~16 &   205    &  207      & 207       \\
& &                                                                                  & 13.1 sec. & 4m 24s   & 4m 13s    & 5m 11s    \\ \cline{3-7}
& & \multirow{2}{*}{CSL, with shift = $-e^{-2\imath\nicefrac{\pi}{17}}=-0.93+0.36\imath$} & 0.45,~~18 &   198    &  204      & 223       \\
& &                                                                                  & 15.4 sec. & 4m 17s   & 4m 19s    & 5m 2s     \\ \cline{2-7}\\[-2.8ex]\cline{2-7}
& \multirow{6}{*}{\begin{sideways} ECS Layers \end{sideways}}
  & \multirow{2}{*}{CSL, with shift = $-1-0.38\imath$}                               & 0.28,~~11 &    142   &    145    & 141         \\
& &                                                                                  & 14.8 sec. & 4m 52s   & 4m 34s    & 4m 56s  \\ \cline{3-7}
& & \multirow{2}{*}{CSG, with angle = $\displaystyle \frac{\pi}{17}$}                & 0.27,~~11 &    139   &    139    & 134        \\
& &                                                                                  & 15.3 sec. & 4m 49s   & 4m 30s    & 4m 48s  \\ \cline{3-7}
& & \multirow{2}{*}{CSL, with shift = $-e^{2\imath\nicefrac{\pi}{17}}=-0.93-0.36\imath$}  & 0.27,~~11 &   138    &   137     & 137         \\
& &                                                                                  & 15.3 sec. & 4m 48s   & 4m 35s    & 5m 1s  \\ \hline \hline \hline
\multirow{12}{*}{\begin{sideways}MP3: $r=150, k=4$\end{sideways}}
& \multirow{6}{*}{\begin{sideways} Sommerfeld BC \end{sideways}}
  & \multirow{2}{*}{CSL, with shift = $-1 + 0.40\imath$}                             & 0.39,~~15 &   681     &   682     &  669       \\
& &                                                                                  & 53.9 sec. & 1h 3m     & 1h 1m 12s & 1h         \\ \cline{3-7}
& & \multirow{2}{*}{CSG, with angle = $\displaystyle-\frac{\pi}{17}$}                & 0.37,~~15 &   620     &  650      & 652        \\
& &                                                                                  & 55.9 sec. & 58m 15s   & 59m 25s   & 1h 2m 12s  \\ \cline{3-7}
& & \multirow{2}{*}{CSL, with shift = $-e^{-2\imath\nicefrac{\pi}{17}}=-0.93+0.36\imath$} & 0.39,~~15 &   623     &  673      & 614        \\
& &                                                                                  & 53.4 sec. & 56m 46s   & 1h 2m     & 55m        \\ \cline{2-7}\\[-2.8ex]\cline{2-7}
& \multirow{6}{*}{\begin{sideways} ECS Layers \end{sideways}}
  & \multirow{2}{*}{CSL, with shift = $-1-0.40\imath$}                               & 0.37,~~15 &    436    &    413    & 423        \\
& &                                                                                  & 1m 26s    & 1h 10m    & 1h 4m     & 1h 8m      \\ \cline{3-7}
& & \multirow{2}{*}{CSG, with angle = $\displaystyle \frac{\pi}{17}$}                & 0.31,~~13 &    390    &    387    & 392        \\
& &                                                                                  & 1m 15s    & 56m 39s   & 55m 46s   & 1h 1m  \\ \cline{3-7}
& & \multirow{2}{*}{CSL, with shift = $-e^{2\imath\nicefrac{\pi}{17}}=-0.93-0.36\imath$}  & 0.30,~~13 &    395    &    390    & 383        \\
& &                                                                                  & 1m 16s    & 57m 55s   & 55m 48s   & 59m 56s  \\ \cline{2-7}
 \end{tabular}
\caption{Experimental results - Iterative solution of two different discrete problems obtained from MP3 (by using different domain sizes and values of $k$). One $V(0,1)$ multigrid cycle is used for approximate inversion of the preconditioners. This preconditioning is used with Bi-CGSTAB, IDR(4) and IDR(8). The interior domain is a square of $r$ units. For $r=90, k=2$ (upper six rows), the problem is discretized with $1024$ interior points along both dimensions. In these problem, when ECS layers are used, there are $256$ additional points per dimension, belonging to the layers. For $r=150, k=4$, the problem is discretized with 2048 interior points per dimension. In the ECS formulation there are 512 extra points (per dimension) in these layers. For both problems, ECS layers are only required, and used, at the north and the east edges of the domain.}}
\label{tab:MP3}
\end{table}

In Table \ref{tab:MP2}, we have solved two test-cases of MP2 with different characterizing parameters. The values of
$\nu$ and $k$ distinguish the test-cases. From Table \ref{tab:meshsizesMP2}, we read that the mesh \edt{width} requirement
for both these test-cases is the same (0.147), and therefore they can be solved on an identical grid of interior size $384^2$.
We observe that although the supremum of the wave number in the domain is identical for both the test-cases, the second
one takes twice the time needed to compute the solution of the first one, for MP2.

\edt{\begin{rem}[Smooth ECS transition]
Plausibly, the sharp rotation of the linear ECS contour may not be very desirable for the discretization of some applications. For this reason, we checked out the iterative performance of the preconditioners and the solvers, for MP2, formulated with an ECS contour that rotates gradually in 256 very small equally sized angles. The one-dimensional analog of the domain is shown in Figure \ref{fig:smoothECS}. The performance of multigrid (on such preconditioners) as well as the Krylov methods turned out very similar to that listed in the tables.
\end{rem}}

The results from experiments on the harder model problem, i.e., MP3 \edt{whose spectrum is more indefinite compared to the other model problems here}, are laid out in Table \ref{tab:MP3}. \edt{We clearly see an advantage in the number of matvecs on problems with the ECS layers against the Sommerfeld BC. However, each matvec of an ECS problem takes longer due to the additional grid points in the layer, hence this advantage is not seen in the CPU time.} Figure \ref{fig:krylov}, details the convergence history of the first test-case of MP3. This is
important in order to check out any possible stagnation trend in the convergence. We find however, that in all the
cases Bi-CGSTAB and \edt{IDR(4) show a well-matched} convergence behavior. \edt{IDR(4), however, seems to be slightly faster in comparison with} Bi-CGSTAB as it avoids the slight initial stagnation evident with Bi-CGSTAB in the figure. With a comparison between
two Krylov methods, it is often also interesting to observe both the CPU time as well as the solver matvecs due to
possible differences in the subspace minimization strategies in different methods. \edt{This reveals that although for many test-runs IDR(8) reports lesser matvecs than IDR(4), it really does not provide any concrete enhancement. Evidently, the reduction in matvecs is easily offset by the increase in CPU time (plus an associated increase in storage which is not shown)}.

% fig: bisom and idrsom and biecs and idrecs
\begin{figure}
\vspace{-1cm}
\centering
\subfigure[Bi-CGSTAB on MP3 with Sommerfeld BC]{\includegraphics[width=7cm]{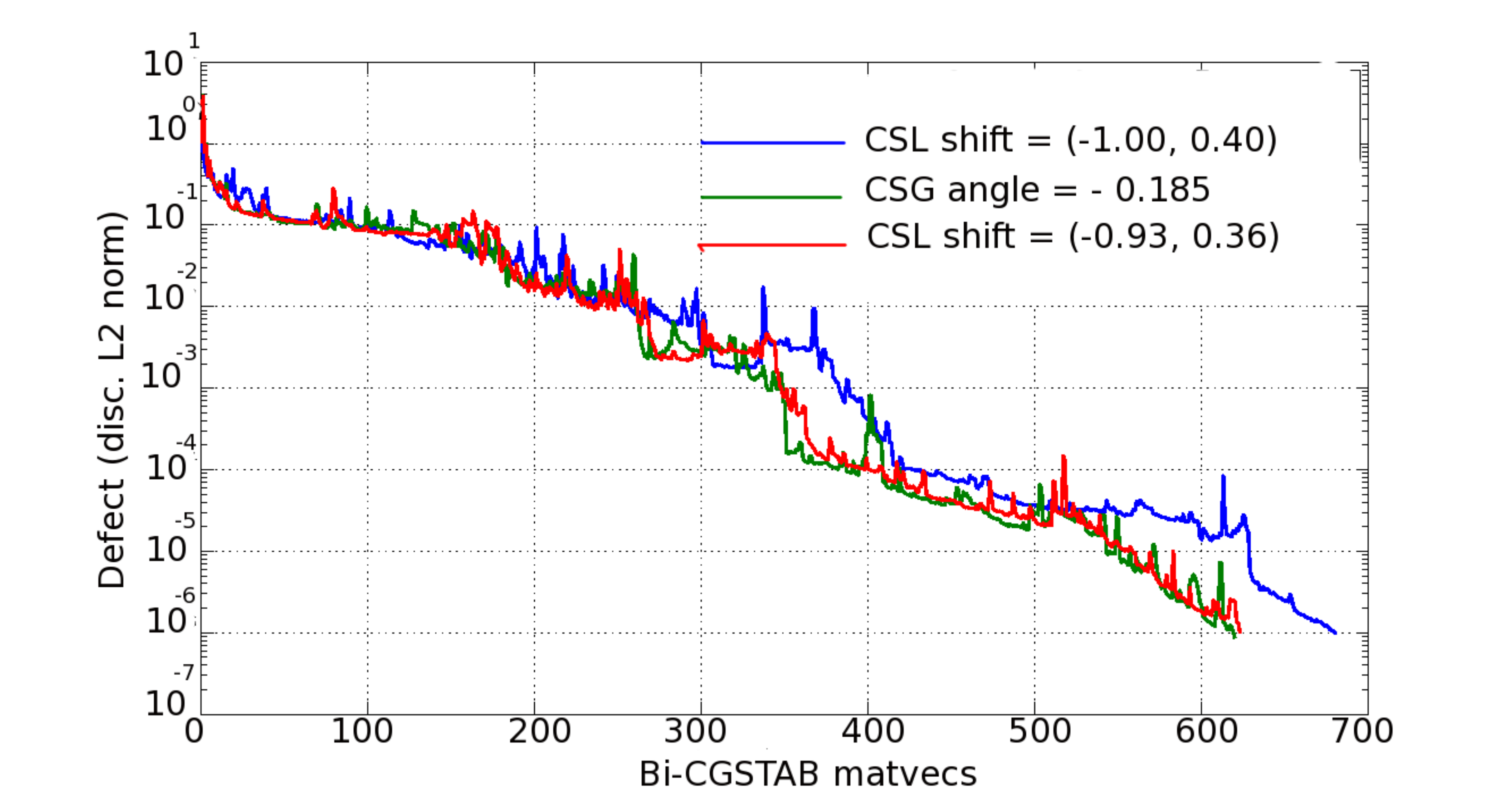}}
\hspace{-5mm} \subfigure[IDR(4) on MP3 with Sommerfeld BC]{\includegraphics[width=7cm]{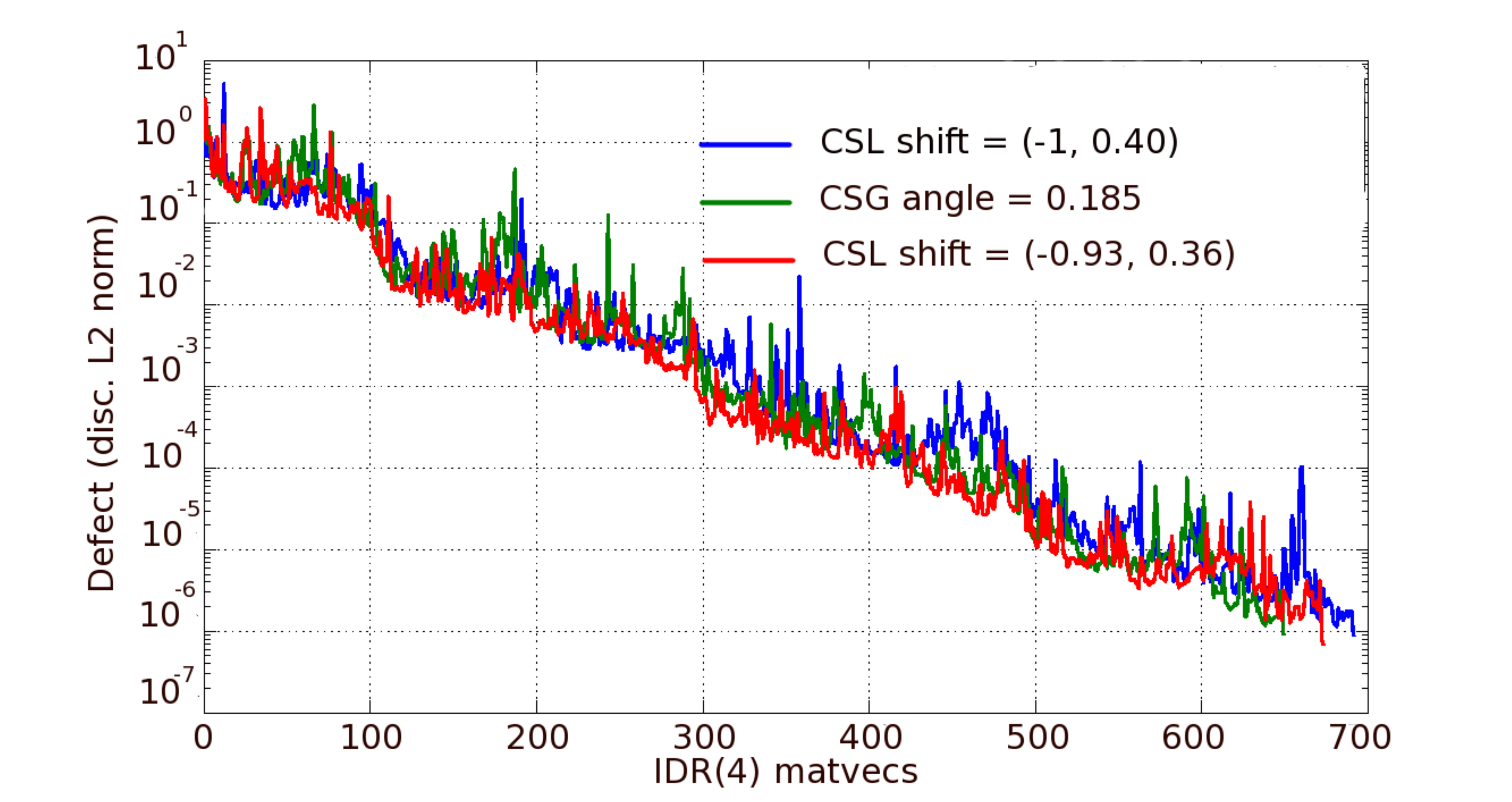}}  \\
\subfigure[Bi-CGSTAB on MP3 with ECS \edt{layers}]{\includegraphics[width=7cm]{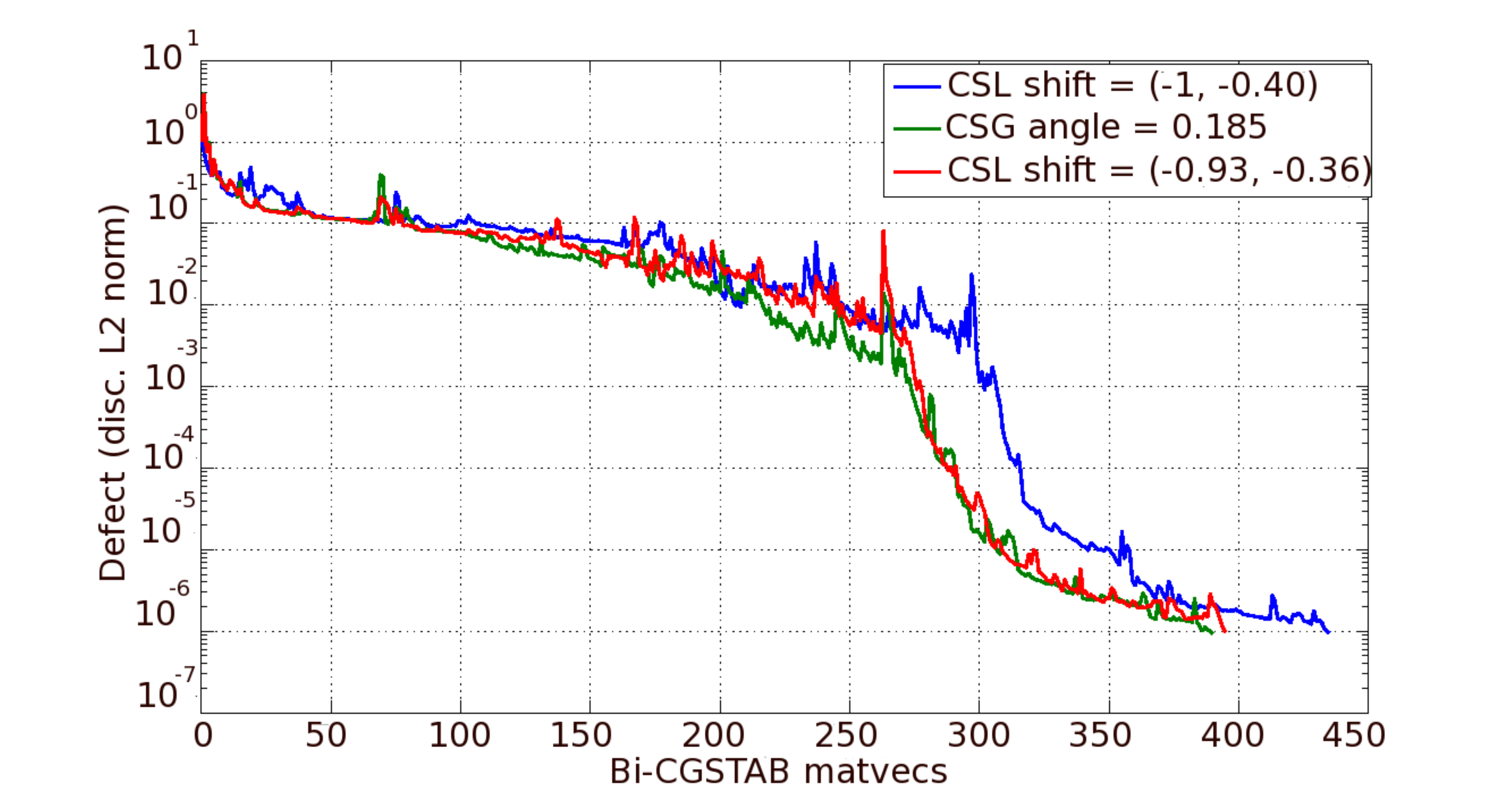}}
\hspace{-5mm} \subfigure[IDR(4) on MP3 with ECS \edt{layers}]{\includegraphics[width=7cm]{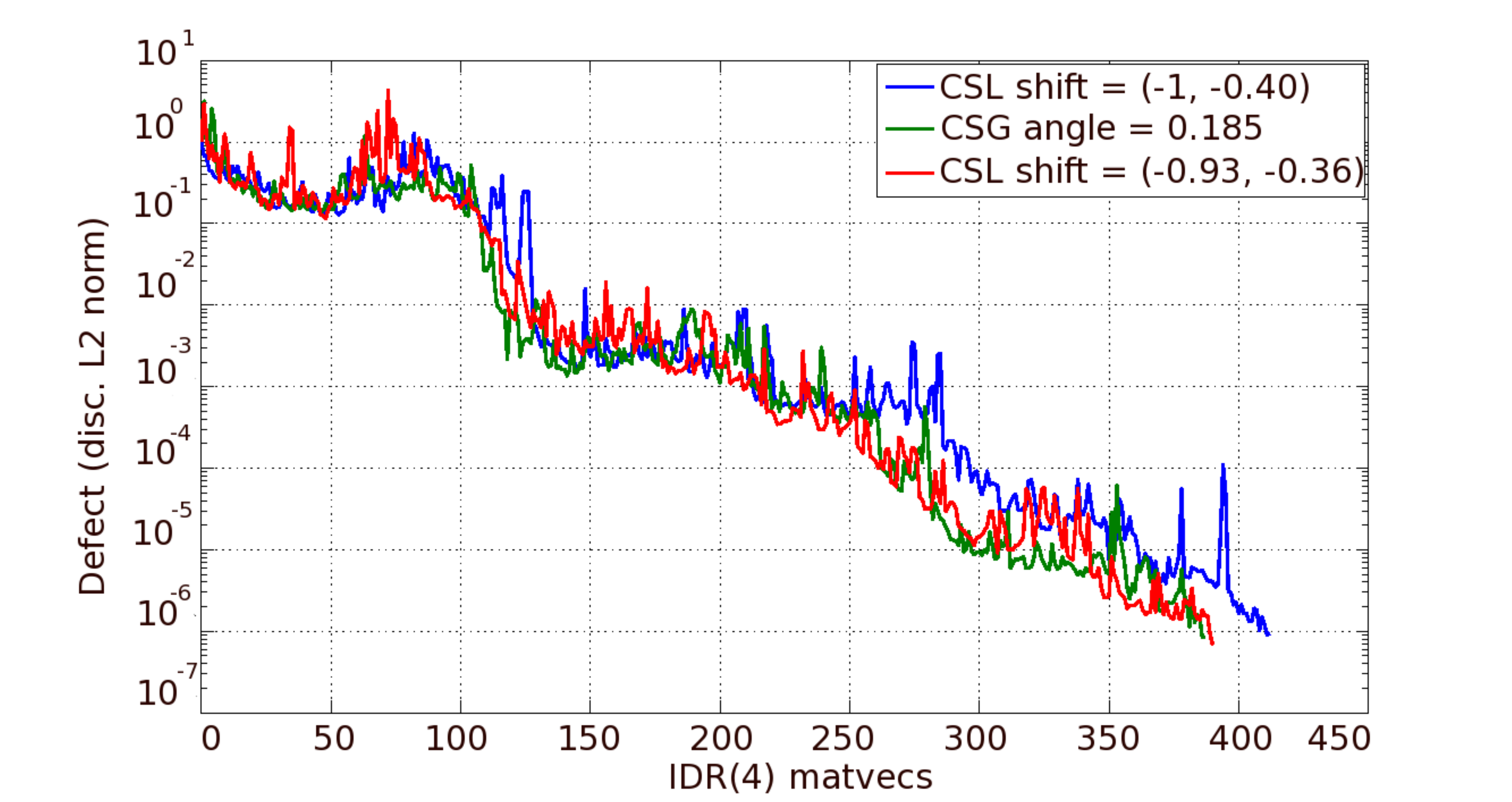}}
\caption{Bi-CGSTAB and IDR(4) defect reduction history with the proposed CSG preconditioner as well as the CSL preconditioner for MP3 with $2048^2$ points in the interior of the domain. ECS \edt{layers} (where used) have added $512$ points per dimension in the exterior layers. Expected convergence traits for Bi-CGSTAB and IDR(4) are clearly visible. (color online)}
\label{fig:krylov}
\end{figure}

We also checked out the solver performance for a larger MP3 problem with size $4096^2$ (not reported in the tables).
The CPU time already runs over 4 hours for a problem of this size, although the performance of the iterative method is
of the same quality as the other problems. This is due to the huge complexity of the test.

\edt{\begin{rem}[Testing the situation with GMRES]
We ran multigrid preconditioned GMRES on MP2, with $\nu=1, k=4$ (ECS formulation). Multigrid was used to approximately invert the CSG preconditioning operator exactly as specified in the second last row of Table \ref{tab:MP2}, during each GMRES iteration. GMRES reported 103 iterations to reduce the relative residual by seven orders of magnitude, and reported the consumed CPU time as 175 seconds. The same level of accuracy can be reached with Bi-CGSTAB or IDR($s$) (as depicted in Table \ref{tab:MP2} in roughly one fifth of the time required for GMRES. GMRES therefore qualifies well for analytic purposes, or where the problem size is small.
\end{rem}}

\edt{\begin{rem}[Platform specification and storage scheme]
These experiments were performed serially on an Intel Xeon (8-core) with 32 gigabytes of RAM. We used Matlab v7 as the testing platform. Some processes such as matvec computations, and backslash inversion of some matrices in the multigrid heirarchy were performed in parallel using all the 8 processors. We implemented some linear algebra routines in mex-C to speed up the computations, and used the \emph{Compressed Row Sparse Format} for storage of the matrix data. Matlab uses \emph{Compressed Column Sparse Format} for such needs. Both the CSL, as well as the CSG preconditioners were stored as sparse matrices in either of these formats.
\end{rem}}

\section{Conclusions and outlook} \label{sec:conclude}
In this paper, we have studied the Helmholtz problems that arise in mathematical models for single and double ionization of atomic and molecular systems. The problems typically have \edt{regions in space where the wave number can be large} and absorbing boundary conditions \edt{are often implemented with complex stretched grids}.

We developed and analyzed the iterative properties of the exterior complex scaled (ECS) absorbing boundary layers for the
indefinite Helmholtz equation. \edt{We have analyzed the spectral properties of the discrete problem formulated with \edt{ECS layers} and found bounds around the spectrum of the Helmholtz operator for constant wave numbers. These bounds were derived for a finite difference Shortley-Weller discretization and linear exterior scaling. Although the theoretical estimates are limited to this model, numerical tests suggest that they are valid for quite general cases.}

An alternative \edt{preconditioner} to the \edt{complex shifted} Laplacian (CSL) is introduced where instead of shifting the wave number, the grids are given a complex scaling. We call this \edt{complex stretched grid} (CSG) preconditioning. We introduced two new benchmark problems that are derived from break-up problems in quantum mechanics and have strongly varying spatially dependent wave numbers. They provide a tough benchmark for future development of iterative Helmholtz solvers.

The CSG and the CSL preconditioners are related and perform similarly for most problems. The preconditioner inversion is performed approximately by a geometric multigrid method, based on ILU(0) relaxation, V(0,1)-cycles, FW restriction, bilinear prolongation, and the Galerkin coarse grid operator. Different numerical experiments with the CSG preconditioner on the model problems, show that our multigrid method is more stable for problems with constant wave numbers. With spatially dependent wave numbers we see that our multigrid method requires a higher damping of the operator to render the indefiniteness manageable. However, with greater damping, the spectra of the preconditioner and the operator get farther apart and this \edt{prolongs} Krylov convergence. Although not listed, we tried different alternatives for multigrid preconditioning, such as F-cycles and more than one V-cycle, however, the best results from a CPU time aspect are shown.

In future, we intend to explore two different avenues \edt{to improve the solver} for \edt{the} model problems \edt{introduced} here. The main
problem is the reduction of CPU time for problems formulated with ECS layers. This can be brought about by first reducing
the complexity of the problem by applying adaptive mesh refinement (AMR) techniques for discretization, say, near the
evanescent layers of the solution. And secondly, by identifying a way of shifting only the most problematic part of the
spectrum, since every cluster of eigenvalues that is needlessly shifted can significantly decrease the performance of the preconditioning. This may results in a better solver than the current state-of-the-art, and is thus our future goal.

\section*{Acknowledgements}
This research was partially supported by FWO-Flanders through grant G.0174.08 as well as by a starting grant from the University of Antwerp, Belgium.

\bibliography{man2_arxiv}

\begin{thebibliography}{10}

\bibitem{BNGM02}
A.~Bogaerts, E.~Neyts, R.~Gijbels, and J.~van~der Mullen.
\newblock Gas discharge plasmas and their applications.
\newblock {\em Spectrochimica Acta Part B: Atomic Spectroscopy},
  57(4):609--658, 2002.

\bibitem{Plasma07}
National Research~Council Plasma 2010~Committe, Plasma Science~Committee.
\newblock {\em Plasma Science: Advancing Knowledge in the National Interest}.
\newblock The National Academies Press, 2007.

\bibitem{RBIM99}
T.N. Rescigno, M.~Baertschy, W.A. Isaacs, and C.W. McCurdy.
\newblock Collisional breakup in a quantum system of three charged particles.
\newblock {\em Science}, 286(5449):2474, 1999.

\bibitem{VMRM05}
W.~Vanroose, F.~Martin, T.~N. Rescigno, and C.~W. McCurdy.
\newblock Complete photo-induced breakup of the h$_2$ molecule as a probe of
  molecular electron correlation.
\newblock {\em Science}, 310:~1787--1789, 2005.

\bibitem{MHR01}
C.W. McCurdy, D.A. Horner, and T.N. Rescigno.
\newblock Practical calculation of amplitudes for electron-impact ionization.
\newblock {\em Physical Review A}, 63(2):22711, 2001.

\bibitem{EM77}
B.~Engquist and B.~Majda.
\newblock Absorbing boundary conditions for the numerical simulation of waves.
\newblock {\em Mathematics of Computation}, 31(139):629--651, 1977.

\bibitem{BT80}
A.~Bayliss and E.~Turkel.
\newblock Radiation boundary conditions for wave-like equations.
\newblock {\em Communications on Pure and Applied Mathematics}, 33:707--725,
  1980.

\bibitem{B94}
J.-P. B\'erenger.
\newblock A perfectly matched layer for the absorption of electromagnetic
  waves.
\newblock {\em Journal of Computational Physics}, 114:185--200, 1994.

\bibitem{G08}
D.~Givoli.
\newblock {\em Computational Acoustics of Noise Propagation in Fluids}, chapter
  Computational absorbing boundaries, pages 145--166.
\newblock Springer, 2008.

\bibitem{CW94}
W.~C. Chew and W.~H. Weedon.
\newblock A 3d perfectly matched medium from modified maxwell's equations with
  stretched coordinates.
\newblock {\em Microwave and Optical Technology Letters}, 7(13):599--604, 1994.

\bibitem{AC71}
J.~Aguilar and J.M. Combes.
\newblock A class of analytic perturbations for one-body schr\"odinger
  hamiltonians.
\newblock {\em Communications in Mathematical Physics}, 22:269--279, 1971.

\bibitem{BC71}
E.~Balslev and J.M. Combes.
\newblock Spectral properties of many-body schr\"odinger operators with
  dilatation-analytic interactions.
\newblock {\em Communications in Mathematical Physics}, 22:280--294, 1971.

\bibitem{EVO04}
Y.A. Erlangga, C.~Vuik, and C.W. Oosterlee.
\newblock On a class of preconditioners for solving the helmholtz equation.
\newblock {\em Applied Numerical Mathematics}, 50:409--425, 2004.

\bibitem{V92}
H.~A. Van~Der Vorst.
\newblock A fast and smoothly converging variant of bi-cg for the solution of
  nonsymmetric linear systems.
\newblock {\em SIAM Journal on Scientific and Statistical Computing},
  13:631--644, 1992.

\bibitem{SG08}
P.~Sonneveld and M.~B. van Gijzen.
\newblock Idr($s$): A family of simple and fast algorithms for solving large
  nonsymmetric systems of linear equations.
\newblock {\em SIAM Journal on Scientific Computing}, 31(2):1035--1062, 2008.

\bibitem{S79}
B.~Simon.
\newblock The definition of molecular resonance curves by the method of
  exterior complex scaling.
\newblock {\em Physics Letters A}, 71:211, 1979.

\bibitem{KP09}
S.~Kim and J.E. Pasciak.
\newblock Analysis of the spectrum of a cartesian perfectly matched layer (pml)
  approximation to acoustic scattering problems.
\newblock {\em Journal of Mathematical Analysis and Applications},
  361:420--430, 2009.

\bibitem{SW38}
G.~H. Shortley and R.~Weller.
\newblock Numerical solution of {L}aplace's equation.
\newblock {\em Journal of Applied Physics}, 9:334--348, 1938.

\bibitem{BGT85}
A.~Bayliss, C.~I. Goldstein, and E.~Turkel.
\newblock On accuracy conditions for the numerical computation of waves.
\newblock {\em Journal of Computational Physics}, 59:~396--404, 1985.

\bibitem{IB95}
F.~Ihlenburg and I.~Babuska.
\newblock Finite element solution to the helmholtz equation with high wave
  numbers.
\newblock {\em Computers and Mathematics with Applications}, 30:9--37, 1995.

\bibitem{EOV06}
Y.A. Erlangga, C.W. Oosterlee, and C.~Vuik.
\newblock A novel multigrid based preconditioner for heterogeneous helmholtz
  problems.
\newblock {\em SIAM Journal on Scientific Computing}, 27:1471--1492, 2006.

\bibitem{EVO06}
Y.A. Erlangga, C.~Vuik, and C.W. Oosterlee.
\newblock Comparison of multigrid and incomplete lu shifted-laplace
  preconditioners for the inhomogeneous helmholtz equation.
\newblock {\em Applied Numerical Mathematics}, 56:648–--666, 2006.

\bibitem{D06}
T.~A. Davis.
\newblock {\em Direct Methods for Sparse Linear Systems (Fundamentals of
  Algorithms 2)}.
\newblock SIAM, Philadelphia, PA, USA, 2006.

\bibitem{Briggs00}
W.L. Briggs, V.E. Henson, and S.F. McCormick.
\newblock {\em A Multigrid Tutorial}.
\newblock SIAM, 2000.

\bibitem{Wess92}
P.~Wesseling.
\newblock {\em An Introduction to Multigrid Methods}.
\newblock Pure and Applied Mathematics. John Wiley \& Sons, 1992.

\bibitem{Trot01}
U.~Trottenberg, C.W. Oosterlee, and A.~Sch\"uller.
\newblock {\em Multigrid}.
\newblock Academic Press, 2001.

\bibitem{Brandt80}
A.~Brandt.
\newblock Stages in developing multigrid solutions.
\newblock In {\em Proceedings of the 2$^{nd}$ international congress on
  numerical methods for engineers}, pages 23--43. Dunod, Paris, 1980.

\bibitem{EEL01}
O.~C.~Ernst H.~C.~Elman and D.~P. O'Leary.
\newblock A multigrid method enhanced by {K}rylov subspace iteration for
  discrete {H}elmholtz equations.
\newblock {\em SIAM Journal on Scientific Computing}, 23:1291--1315, 2001.

\end{thebibliography}
\end{document}